\documentclass[11pt,twoside,a4paper]{amsart}
\usepackage[english]{babel}

\usepackage[utf8]{inputenc}
\usepackage[T1]{fontenc}

\usepackage{pstricks-add}
\usepackage{array}
\usepackage{amsfonts}
\usepackage{amssymb}
\usepackage{amsmath}
\usepackage{graphicx}
\usepackage{graphics}
\usepackage{amsthm}

\newtheorem{theo}{Theorem}[section]
\newtheorem{de}[theo]{Definition}
\newtheorem{prop}[theo]{Proposition}
\newtheorem{lem}[theo]{Lemma}

\newtheorem{cor}[theo]{Corollary}

\normalsize
\title{On embedding of repetitive Meyer multiple sets into model multiple sets}
\date{\today}

\author{Jean-Baptiste Aujogue}

\begin{document}

\maketitle

\begin{abstract} Model sets are always Meyer sets but the converse is generally not true. In this work we show that for a repetitive Meyer multiple sets of $\mathbb{R}^d$ with associated dynamical system $(\mathbb{X}, \mathbb{R}^d)$, the property of being a model multiple set is equivalent for $(\mathbb{X}, \mathbb{R}^d)$ to be almost automorphic. We deduce this by showing that a repetitive Meyer multiple set can always be embedded into a repetitive model multiple set having a smaller group of topological eigenvalues.
\end{abstract}

\section*{Outline}

In this paper we address a study of particular point patterns of an Euclidean space. From a general point of view, a point pattern is a collection of points inside some space $\mathbb{R}^d$ (or some locally compact Abelian group), which obeys some discreteness and relative density properties. In this work we will be turned onto the slight generalization of \textit{multiple point patterns}, that is, a finite collection of point patterns which may overlap. Thanks to this minor generalization our results become true for symbolic sequences and arrays as well. We precise that the concept of point patterns can be much further generalized, by considering for instance weighted Dirac combs as done in \cite{BaaMo}, or the even more general concept of translation-bounded measures \cite{BaaLe}.



\vspace{0.2cm}
In the topic of point patterns, probably the most understood and studied objects are the so-called \textit{model sets}, or model multiple set according to our point of view. Formally, a model multiple set is a point set of an Euclidean space that arises as follows: we consider a 'total space' to be the product of $\mathbb{R}^d$ (the ambient space) with a locally compact Abelian group $H$ (the internal group), together with a lattice $\tilde{\Gamma }$ in this product. A model multiple set is then a finite family $\left( \Lambda _i\right) _{i\in I}$ of patterns, each $\Lambda _i$ being obtained as projections on $\mathbb{R}^d$ of the points of the lattice whose projection in $H$ falls into some compact and topologically regular subset $ W_i $ of $H$. The data of the group $H$ together with the lattice $\tilde{\Gamma }$ is called a cut $\& $ project scheme and the family $\left\lbrace W_i \right\rbrace _{i\in I}$ is called a window. Any \textit{coding of a rotation} is a model multiple set: such arrays may be defined as the model multiple sets associated with cut $\&$ project schemes and window $\left\lbrace W_i \right\rbrace _{i\in I}$ obeying the following three conditions:

\vspace{0.2cm}
(i) The internal space $H$ is compact

(ii) The window $\left\lbrace W_i \right\rbrace _{i\in I}$ covers $H$

(iii) The sets $W_i$ have pairwise disjoint interiors.
\vspace{0.2cm}

It is also often required that the window admits a trivial redundancy subgroup. Famous examples of such arrays are the \textit{Toeplitz sequences} and arrays \cite{Do}, also called elsewhere \textit{limit-periodic point sets} \cite{BaaMoSch}, which are uniquely characterized as being the coding of a rotation over an \textit{odometer} $H$.


\vspace{0.2cm}

A great challenge for point patterns and symbolic sequences is to characterize intrinsic properties of the considered pattern in therms of a natural dynamical system attached to it. Such a dynamical system is provided, given a multiple point pattern $\Lambda$, by a compact metric space $\mathbb{X}_\Lambda$ called the \textit{hull} of $\Lambda$, together with a natural action of $\mathbb{R}^d$, the space into which $\Lambda$ fits, by homeomorphisms. It turned out that for symbolic sequences, and later generalized for arrays \cite{Pa} \cite{BeSiYi}, the property to be the coding of a rotation is characterizable in such a way: An array $\Lambda$ in $\mathbb{R}^d$ is the coding of a rotation if and only if its associated dynamical system $(\mathbb{X}_\Lambda , \mathbb{R}^d)$ is \textit{almost automorphic} (An equivalent statement can be set on the subshift $\Xi _\Lambda$ generated by the array, with the group $\mathbb{Z}^d$ acting on it \cite{BeSiYi}). This property deals with a certain factor map 
\begin{align*}
\begin{psmatrix}[colsep=1.5cm,
rowsep=1.2cm]
\pi : \mathbb{X}_\Lambda \; \;  & \; \; \mathbb{X}_{eq}
\psset{arrows=->>,linewidth=0.2pt, labelsep=1pt
,nodesep=0pt}
\ncline{1,1}{1,2}
\end{psmatrix}
\end{align*} where the factor space $\mathbb{X}_{eq}$ is called the \textit{maximal equicontinuous factor} of $(\mathbb{X}_\Lambda , \mathbb{R}^d)$. For any point pattern, as well as for any topologically transitive dynamical system over a compact space with Abelian acting group, it is a compact Abelian group. A dynamical system is called almost automorphic whenever its factor map $\pi$ admits a one-point fiber \cite{Vee}.


\vspace{0.2cm}

If one wishes to provide such a characterization form arrays to general multiple point patterns, one has to consider a property, which to our knowledge doesn't admit any counterparts on the dynamical system $(\mathbb{X}_\Lambda , \mathbb{R}^d)$ (see however \cite{KeSa}), namely the so-called \textit{Meyer property}. This property admits several different but equivalent forms (\cite{KeSa},\cite{La},\cite{Mo}), and is satisfied by model multiple sets. Then a powerful theorem of Baake, Lenz and Moody \cite{BaaLeMo}, generalized in the multiple set setting in \cite{LeeMo}, asserts that a Meyer set $\Lambda$ is a regular model set if and only if its associated dynamical system is almost automorphic with the one-point fiber elements of $\mathbb{X}_{eq}$ having full Haar measure. Here the regularity property means that each $W_i$ used to construct the considered model multiple set has a boundary of Haar measure zero in $H$. However, this result gets rid of the class of non necessarily regular model multiple sets, and thus of many interesting examples.

\vspace{0.2cm}

The strategy given in \cite{BaaLeMo} remains on the use of the so-called \textit{autocorrelation hull} $\mathbb{A}_\Lambda$ of a point pattern \cite{MoSt} \cite{Mo2} \cite{BaaMo} \cite{LeeMo}, in a context where it is compact and a factor of the original hull $\mathbb{X}_\Lambda$. This is not necessary satisfied for model sets which are not regular. On one hand regular model sets are all uniquely ergodic, with pure point dynamical spectrum (with respect to the unique ergodic probability measure) and of entropy zero. On the other hand one can find examples of model sets, apprearing as subsets of lattices or equivalently as $0-1$ arrays, with an arbitrary large positive topological entropy \cite{MaPa} (such examples can even be chosen uniquely ergodic \cite{BuKw}). By a result of \cite{BaaLeRi} these examples cannot have a pure point dynamical spectrum, which implies by the characterization of purely diffractive point patterns provided in \cite{BaaMo} that the autocorrelation hull cannot be compact, and certainly not a factor of $\mathbb{X}_\Lambda$. This observation shows that the techniques developped in \cite{BaaLeMo} does not generalize to not necessarily regular model sets.
Still, one has the following:

\vspace{0.3cm}
\textbf{Theorem 1} \textit{A repetitive Meyer multiple set $\Lambda $ of $\mathbb{R}^d$ is a model multiple set if and only if its associated dynamical system $(\mathbb{X}_\Lambda, \mathbb{R}^d)$ is almost automorphic.}
\vspace{0.3cm}

The method developped here greatly differs from the one of \cite{BaaLeMo}. It is in particular no question of autocorrelation, and to some extent isn't related to any measure aspect. The earlier statement will in fact shows up as a consequence of a more general fact concerning the embedding of Meyer multiple sets into model multiple sets. A Meyer set of $\mathbb{R}^d$ is always a subset of some model set \cite{Mo}. We revisit this fact here in the case of repetitive Meyer multiple sets:


\vspace{0.3cm}
\textbf{Theorem 2} 
Let $\Lambda _0$ be a repetitive Meyer multiple sets of $\mathbb{R}^d$ with dynamical system $(\mathbb{X}_{\Lambda _0} , \mathbb{R}^d)$ and factor map $\pi : \mathbb{X}_{\Lambda _0} \twoheadrightarrow \mathbb{X}_{eq}$ onto its maximal equicontinuous factor.

$(i)$ There is a cut $\&$ project scheme $(H , \Sigma , \mathbb{R}^d)$ with window $\left\lbrace  W_i \right\rbrace _{i\in I}$ in the internal space, such that $H \times _\Gamma \mathbb{R}^d \simeq \mathbb{X}_{eq}$ as compact Abelian groups. We denote $\mathbb{X}_{MS}$ to be the unique hull of repetitive model multiple sets issuing from this cut $\&$ project scheme and window.

\vspace{0.2cm}
$(ii)$ Any $\Lambda $ in $\mathbb{X}_{\Lambda _0}$ symbol-wise embeds in a repetitive model multiple set $\Delta _\Lambda $ in $\mathbb{X}_{MS}$. Moreover, For residually many $\Delta \in \mathbb{X}_{MS}$ there exist a subset $ C_\Delta $ of $ \mathbb{X}_{\Lambda _0}$ such that
\begin{align*} \Delta _i= \bigcup_{\Lambda \in C_\Delta } \Lambda _i
\end{align*}

$(iii)$ The group of topological eigenvalues $\mathcal{E}(\mathbb{X}_{MS},\mathbb{R}^d)$ is a subgroup of $\mathcal{E}(\mathbb{X}_{\Lambda _0},\mathbb{R}^d)$.

\vspace{0.3cm}

Theorem $1$ is a very direct consequence of theorem $2$, and consist in showing that the dynamical system $(\mathbb{X}_{\Lambda _0} , \mathbb{R}^d)$ is almost automorphic if and only if the hulls $\mathbb{X}_{\Lambda _0}$ and $\mathbb{X}_{MS}$ are the same collection of point patterns. The question whether the dynamical systems $(\mathbb{X}_{MS},\mathbb{R}^d)$ and $(\mathbb{X}_{\Lambda _0},\mathbb{R}^d)$ have the same group of topological eigenvalues can be addressed in terms of the redundancy subgroup $\mathcal{R}$ of the multiple window: This is a compact subgroup of $H$ whose pontryagin dual is precisely equal to the quotient
\begin{align*} \widehat{\mathcal{R}}= \mathcal{E}(\mathbb{X}_{\Lambda _0},\mathbb{R}^d)\diagup \mathcal{E}(\mathbb{X}_{MS},\mathbb{R}^d)
\end{align*}

It is also interesting to note that, given $\Lambda _0$ a repetitive Meyer set, any $\Delta _{\Lambda _0}\in \mathbb{X}_{MS}$ containing $\Lambda _0$ generates a subgroup $\langle \Delta _{\Lambda _0}\rangle$ equal to the subgroup $\langle \Lambda _0\rangle$ generated by $\Lambda _0$. This means that if $\Lambda _0$ is not supported on a lattice in $\mathbb{R}^d$ then $\Delta _{\Lambda _0}$ is not supported in a lattice as well, and in particular cannot be fully periodic.

\vspace{0.4cm}

This paper is organized as follows: in first section we make a presentation of point patterns and their associated dynamical systems, and introduce the so-called model sets. We provide in section $2$ a short survey on the necessary notions of topological dynamics which are involved in this article. In section $3$ we present a new topology which we call combinatoric topology, and show that in case of a repetitive Meyer multiple set this topology naturally gives rise to a locally compact Abelian group. In section $4$ we briefly introduce what we call there a subsystem of the hull of a point pattern, and from the results of sections $3$ and $4$ we provide in section $5$ a construction of a cut $\&$ project scheme with multiple window. We provide in section $6$ the statement and the proof of theorem $2$, and we deduce the proof of theorem $1$ in section $7$.

\clearpage

\section{Point patterns and their dynamical systems}

\subsection{Point patterns and the Meyer property}

Let us consider a finite collection $\Lambda = \left( \Lambda _i\right) _{i\in I}$ of point sets in an Euclidean space $\mathbb{R}^d$, whose support is set as the finite union $\underline{\Lambda}:= \bigcup _{i\in I} \Lambda _i$. We say that $\Lambda$ is \textit{uniformly discrete} if each open ball of some radius $r_0>0$ in $\mathbb{R}^d$ contains at most one point of $ \underline{\Lambda }$. We say that $\Lambda$ is \textit{relatively dense} if each open ball of some radius $R_0$ contains at least one point of each subset $\Lambda _i$. We then say that $\Lambda$ is a \textit{Delone} multiple set whenever it is both uniformly discrete and relatively dense in $\mathbb{R}^d$. Such multiple sets will also be refered along these lines as \textit{point patterns}.

\vspace{0.2cm}
An important notion affiliated with a point pattern $\Lambda$ is its \textit{language}, that is, the collection of all 'circular-shaped' patterns appearing at some site of $\Lambda$
\begin{align*} \mathcal{L}_\Lambda := \left\lbrace (\Lambda -\gamma )\cap B(0,R) \; \vert \; \gamma \in \underline{\Lambda} , \; R>0 \right\rbrace 
\end{align*}
A point pattern $\Lambda _0$ is said to have \textit{finite local complexity} whenever it has only a finite number of circular patterns for any fixed radius. The language of a point pattern $\Lambda _0$ can be used to set a collection of other point patterns called its \textit{hull}, which is defined as
\begin{align*} \mathbb{X }_{\Lambda _0}:= \left\lbrace \Lambda  \subset \mathbb{R}^d \, \vert \; \mathcal{L}_{\Lambda }\subseteq \mathcal{L}_{\Lambda _0} \right\rbrace 
\end{align*}
Thus, from its very construction the hull of $\Lambda _0$ is nothing but the collection of point sets whose bounded patterns appears somewhere in $\Lambda _0$. It admits a natural topology sometimes called the \textit{local topology} which arise from metric
\begin{align*} d(\Lambda , \Lambda '):= \inf \left\lbrace \dfrac{1}{1+R} \; \vert \; \exists \vert t \vert , \vert t'\vert < \dfrac{1}{1+R} , \; (\Lambda -t)\cap B(0,R) = (\Lambda '-t')\cap B(0,R) \right\rbrace 
\end{align*}
This metric roughly means that two point sets are close if they agree on a large domain about the origin up to small shifts. Now if we are given a point set $\Lambda$ then any vector $t\in \mathbb{R}^d$ defines a point set $\Lambda .t$ by simply shifting any site of $\Lambda$ by $-t$. Obviously if some point set lies within the hull $\mathbb{X}_{\Lambda _0}$ then any of its translates remains in $\mathbb{X}_{\Lambda _0}$, providing so a jointly continuous action of $\mathbb{R}^d$ on $\mathbb{X}_{\Lambda _0 }$ and thus a dynamical system $(\mathbb{X}_{\Lambda _0}, \mathbb{R}^d)$. It can be observed that the space $\mathbb{X}_{\Lambda _0}$ topologized in the above way is nothing but the completion of the $\mathbb{R}^d$-orbit of $\Lambda _0$ with respect to the metric $d$. In particular the dynamical system $(\mathbb{X}_{\Lambda _0}, \mathbb{R}^d)$ is topologically transitive.

\vspace{0.2cm}

We consider the \textit{canonical transversal} of the hull of $\Lambda _0$ to be the subset
\begin{align}\label{canonical.transversal} \Xi := \left\lbrace \Lambda   \in \mathbb{X}_{\Lambda _0} \, \vert \, 0\in \underline{\Lambda }\right\rbrace 
\end{align}
The terminology comes from the fact that this set intersects each orbit under the $\mathbb{R}^d$-action, and that any radius of uniform discreteness $r_0$ for $\underline{\Lambda _0}$ is such that if $x\in \Xi $ then $x.t \notin \Xi $ for any $t\in B(0,r_0)$. Now using this set we can turn the finite local complexity condition into the topological setting:
\vspace{0.2cm}
\begin{prop} Let $\Lambda$ be a point pattern of $\mathbb{R}^d$ with hull $\mathbb{X}_{\Lambda _0}$. The assertions are equivalent:

(i) $\Lambda _0$ is of finite local complexity.

(ii) The set of differences $\underline{\Lambda}_0 -\underline{\Lambda}_0 := \left\lbrace \gamma -\gamma '\, \vert \gamma ,\gamma '\in \underline{\Lambda}_0 \right\rbrace $ is discrete and closed in $\mathbb{R}^d$.

(iii) The metric space $\mathbb{X}_{\Lambda _0} $ is compact.

(iv) The transversal $\Xi $ with induced topology is compact.
\end{prop}

\vspace{0.2cm}
All of this follows from the fact that finite local complexity means exactly the precompacity of the transversal $\Xi$ with respect to the metric $d$. A Delone multiple set $\Lambda $ of finite local complexity is \textit{repetitive} if for all $t_0\in \mathbb{R}^d$ and $R>0$ the set
of $t\in \mathbb{R}^d$ such that $(\Lambda +t)\cap B(t_0,R)\equiv \Lambda \cap B(t_0,R) $ is relatively dense in $\mathbb{R}^d$. This notion is equivalent to the minimality of the dynamical system $(\mathbb{X}_{\Lambda }, \mathbb{R}^d)$.

\vspace{0.2cm}
\begin{de}\label{defmeyer}\cite{La} A Delone multiple set $\Lambda $ admits the Meyer property if the set of differences $\underline{\Lambda} -\underline{\Lambda} $ is uniformly discrete in $\mathbb{R}^d$, or equivalently if there exists a finite set $F$ with $\underline{\Lambda} -\underline{\Lambda} \subset \underline{\Lambda}+F$.
\end{de}

\vspace{0.2cm}
The Meyer property on patterns admits several different, but equivalent, formulations (\cite{La},\cite{Mo}). As the above definition suggests, this property appears as a strengthen version of the finite local complexity property. In fact we have even more:

\vspace{0.2cm}
\begin{prop}\label{prop:meyer}\cite{Mo} If $\Lambda $ is a Meyer multiple set then all finite combinations $\underline{\Lambda} \pm \underline{\Lambda} \pm .. \pm \underline{\Lambda}$ (with any choice of sign) is uniformly discrete.
\end{prop}
\vspace{0.2cm}
\subsection{Model sets} Now we specify our attention to the so-called \textit{Model sets} of $\mathbb{R}^d$. 
\subsubsection{Cut $\&$ project schemes and window}
\begin{de}\label{CPS} A cut $\&$ project scheme is a data $(H, \Sigma, \mathbb{R}^d)$, where $H$ is a locally compact Abelian group, with some diagram of the form
\begin{align*}\begin{psmatrix}[colsep=1.5cm,
rowsep=0cm]
H \; \; & \; \; H\times \mathbb{R}^d \; \; & \; \; \mathbb{R}^d \\
\cup & \cup & \cup \\
\Gamma ^* \; \; & \; \;  \Sigma\; \; & \; \; \Gamma 
\psset{arrows=->>,linewidth=0.2pt, labelsep=1.5pt
,nodesep=0pt}
\ncline{1,2}{1,3}
\psset{arrows=<->,linewidth=0.2pt, labelsep=1.5pt
,nodesep=0pt}
\ncline{3,2}{3,3}
\psset{arrows=<<-,linewidth=0.2pt, labelsep=1.5pt
,nodesep=0pt}
\ncline{1,1}{1,2}
\ncline{3,1}{3,2}
\end{psmatrix}\end{align*}
such that:\\
- $\Sigma$ is a discrete and co-compact subgroup of $H\times \mathbb{R}^d$.\\
- the canonical projection onto $\mathbb{R}^d$ is bijective from $\Sigma$ to its image $\Gamma $.\\
- the image $\Gamma ^*$ of $\Sigma$ under the canonical projection onto $H$ is a dense subgroup of $H$.
\end{de}

\vspace{0.2cm}
There is a well-established formalism associated to a cut $\&$ project scheme: the space $\mathbb{R}^d$ is usually called the \textit{physical space} and $H$ the \textit{internal space}, and the subgroup $\Gamma $ of $\mathbb{R}^d$ is called the \textit{structure group}. Also, the group morphism $\Gamma \longrightarrow H$ which to any $\gamma $ associates $\gamma ^*:= \pi _{H}(\pi _{\mathbb{R}^d}^{-1}(\gamma ))\in \Gamma ^*$ is usually called the \textit{*-map} of the cut $\&$ project scheme, whose graph is precisely the subgroup $\Sigma$.

\vspace{0.2cm}
\begin{de}\label{defwindow} Let $(H, \Sigma, \mathbb{R}^d)$ be a cut $\&$ project scheme. A window is a finite collection $\left\lbrace  W_j \right\rbrace _{j\in I}$ of compact topologically regular subsets of $H$ (a subset of a topological space is topologically regular if it is the closure of its interior), subject to the condition that the subset $NS^H:= H\backslash \left[ \Gamma ^*- \bigcup _{i\in I}\partial W_i \right]$ is non-empty in $H$.
\end{de}
\vspace{0.2cm}
Note that if $\Gamma $ is countable then $NS^H$ is automatically non-empty from the Baire category theorem, but otherwise one has to assume its non-emptiness.

\vspace{0.2cm}
\subsubsection{Model sets} Given a cut $\&$ project scheme $(H, \Sigma, \mathbb{R}^d)$ with window $\left\lbrace  W_i \right\rbrace _{i\in I}$ in $H$, we consider a \textit{model multiple set} to be the point pattern $\left( \mathfrak{P}(W_i)\right)  _{i\in I}$ of $\mathbb{R}^d$ where $$\mathfrak{P}(W_i):= \left\lbrace \gamma \in \Gamma \; \vert \; \gamma ^*\in W_i \right\rbrace $$
We may also consider translates of the resulting point pattern by any vector $t\in \mathbb{R}^d$, or translates of the window $\left\lbrace  W_j \right\rbrace _{j\in I}$ by any element $w\in H$, which in both cases leads to a new point pattern of $\mathbb{R}^d$.

\vspace{0.2cm}
\begin{de}\label{defmodelset} An inter-model multiple set $\Lambda $ associated with a cut $\&$ project scheme $(H, \Sigma, \mathbb{R}^d)$ with a window $\left\lbrace  W_i \right\rbrace _{i\in I}$ in $H$ is a point pattern of $\mathbb{R}^d$ of the form
\begin{align*} \mathfrak{P}( \stackrel{\circ }{W_i}+w) -t \subseteq  \Lambda _i \subseteq \mathfrak{P}(W_i+w) -t
\end{align*}
\end{de}
\vspace{0.2cm}
In the particular case where the window admits boundary sets $ \partial W_i$ of null Haar measure in $H$, any resulting inter-model multiple set is said to be \textit{regular}. We will don't assume this property in what follows. An inter-model multiple set is called \textit{non-singular} or sometimes \textit{generic} when there are equalities
\begin{align}\label{generic.model.set} \mathfrak{P}( \stackrel{\circ }{W_i}+w) -t =  \Lambda _i = \mathfrak{P}(W_i+w) -t
\end{align}
The situation where such equality occurs for a given couple $(w,t)$ clearly only depends on the choice of $w\in H$, and happens precisely when there are no point of $\Gamma ^*$ within any boundary $w+\partial W_i$: In other words equality occurs exactly when $w\in NS^H$. The next statement is folklore in the topic of model sets:

\vspace{0.2cm}
\begin{theo}\label{theo:théoreme.élémentaire.CP} Let $(H,\Sigma,\mathbb{R}^d)$ be a cut $\&$ project scheme with window $\left\lbrace  W_i \right\rbrace _{i\in I}$ in $H$. Consider the full family $\mathbb{X}_{(H,\Sigma,\mathbb{R}^d), \left\lbrace  W_i \right\rbrace _{i\in I}}$ of inter-model multiple sets that emerge from this data. Then:

(i) $\mathbb{X}_{(H,\Sigma,\mathbb{R}^d), \left\lbrace  W_i \right\rbrace _{i\in I}}$ is complete for the metric $d$, $\mathbb{R}^d$-invariant, and all point pattern of this family has the Meyer property.

(ii) The class of non-singular inter-model multiple sets generates a unique hull, denoted $\mathbb{X}_{MS}$, included into $\mathbb{X}_{(H,\Sigma,\mathbb{R}^d),\left\lbrace  W_i \right\rbrace _{i\in I}}$.

(iii) An inter-model multiple set of $\mathbb{X}_{(H,\Sigma,\mathbb{R}^d),\left\lbrace  W_i \right\rbrace _{i\in I}}$ is repetitive if and only if it lies within $\mathbb{X}_{MS}$.
\end{theo}
\vspace{0.2cm}
In the sequel we will mainly focus on repetitive point patterns, and thereby in the context of inter-model multiple sets on the unique minimal hull arising from a given cut $\&$ project scheme and window. One could equivalently define $\mathbb{X}_{MS}$ to be the hull of any given non-singular model multiple set arising from the data.

\vspace{0.2cm}
\subsubsection{Redundancies} It is possible, given a cut $\&$ project scheme $(H,\Sigma,\mathbb{R}^d)$ with window $\left\lbrace  W_i \right\rbrace _{i\in I}$, to tighten up the data without altering the repetitive model multiple sets which may emerge. We call the \textit{redundancies group} of the multiple window the subgroup of $H$ given by
\begin{align*} \mathcal{R}:= \left\lbrace w\in H \, \vert \, W_i+w=W_i \, \forall \,i\in I \right\rbrace 
\end{align*}
and we say that a multiple window is \textit{irredundant} is its associated redundancies group is trivial. In any cases, $\mathcal{R}$ is a closed subgroup of $H$, and is moreover compact, as any redundancy satisfies $w\in W_i-W_i$, this latter being compact. Thus we may consider the quotient locally compact Abelian group $ H_\mathcal{R}$. Then as argumented in \cite{LeeMo}, the data $(H_\mathcal{R}, \Sigma _\mathcal{R}, \mathbb{R}^d)$ with associated diagram
\begin{align*}\begin{psmatrix}[colsep=1.5cm,
rowsep=0cm]
H_\mathcal{R} \; \;  & \; \; H_\mathcal{R} \times \mathbb{R}^d \; \; & \; \; \mathbb{R}^d \\
\cup & \cup & \cup \\
\left[ \Gamma ^*\right] _\mathcal{R} \; \; & \; \; \Sigma _\mathcal{R} \; \; & \; \; \Gamma 
\psset{arrows=->>,linewidth=0.2pt, labelsep=1.3pt
,nodesep=0pt}
\ncline{1,2}{1,3}
\psset{arrows=<->,linewidth=0.2pt, labelsep=1.3pt
,nodesep=0pt}
\ncline{3,2}{3,3}
\psset{arrows=<<-,linewidth=0.2pt, labelsep=1.3pt
,nodesep=0pt}
\ncline{1,1}{1,2}
\ncline{3,1}{3,2}
\end{psmatrix}\end{align*}
where $\Sigma _\mathcal{R}:= \left\lbrace ([\gamma ^*]_\mathcal{R}, \gamma ) \in H_\mathcal{R} \times \mathbb{R}^d \; \vert \, \gamma \in \Gamma \right\rbrace $, is a cut $\&$ project scheme and $\left\lbrace [W_i]_\mathcal{R}\right\rbrace _{i\in I} $ is an irredondant window in the internal space. Using the equality $W_i +\mathcal{R} = W_i$ for each index $i\in I$, one may remark that any inter-model multiple set of the latter data arise as an inter-model multiple set of the former, but the converse is not true in general. However, it is true for the non-singular model multiple sets: Due to the equality $W_i +\mathcal{R} = W_i$ one has $\partial W_i +\mathcal{R} = \partial W_i$, so that a $w\in H$ lies within $NS^H$ if and only if its class $[w]_\mathcal{R}$ lies within $NS^{H_\mathcal{R}}$. Then for each such element we have
\begin{align*} \mathfrak{P}(W_i +w)-t = \mathfrak{P}([W_i]_\mathcal{R} +[w]_\mathcal{R})-t
\end{align*}
that is, the cut $\&$ project scheme with irredundant window gives rise to the same collection of non-singular model multiple sets. It directly follows that it gives rise to the same hull of repetitive inter-model multiple sets.

\vspace{0.2cm}

\subsubsection{Parametrization map for model sets}
Given a cut $\&$ project scheme $(H,\Sigma,\mathbb{R}^d)$ with window $\mathfrak{W}:= \left\lbrace  W_i \right\rbrace _{i\in I}$, any element $w\in NS^H$ and vector $t\in \mathbb{R}^d$ defines a unique model multiple set according to equalities (\ref{generic.model.set}). If moreover after possibly moding out some subgroup of the internal space the considered window is irredundant, then there is only one possible pair $(w,t)\in H\times \mathbb{R}^d$ yielding a given model set, up to an element of $\Sigma$. Thus one would ask for the existence of some mapping from $H\times _\Sigma \mathbb{R}^d$, the compact Abelian group obtained as the quotient of $H\times \mathbb{R}^d$ by the lattice $\Sigma$, into the space of point patterns $\mathbb{X}_{(H,\Sigma,\mathbb{R}^d), \left\lbrace  W_i \right\rbrace _{i\in I}}$. It is in fact the inverse phenomenon which occurs (we restrict ourselves to the space $\mathbb{X}_{MS}$ of repetitive inter-model multiple sets, see \cite{LeeMo} for a more general version):


\vspace{0.2cm}
\begin{theo}\label{theo:parametrization.map}\cite{Sch} \cite{LeeMo} Let $\mathbb{X}_{MS}$ be the hull of repetitive inter-model multiple sets arising from a cut $\&$ project scheme $(H ,\Sigma,\mathbb{R}^d)$ and window $\left\lbrace W_i\right\rbrace _{i\in I}$, which can be supposed irredundant. Then there exists a factor map
\begin{align*}
\begin{psmatrix}[colsep=1.5cm,
rowsep=0cm]
 \mathbb{X}_{MS}\; \; & \; \;  H\times _{\Gamma } \mathbb{R}^d 
\psset{arrows=->>,linewidth=0.2pt, labelsep=1.5pt
,nodesep=0pt}
\ncline{1,1}{1,2}
\end{psmatrix}
\end{align*}
which satisfies $\pi (\Lambda )= [w_\Lambda , t_\Lambda ]_{\Sigma}$ if and only if there are inclusions
\begin{align*}
\mathfrak{P} (\mathring{W_i}+w_\Lambda )-t_\Lambda \subseteq \Lambda_i \subseteq \mathfrak{P} (W_i+w_\Lambda )-t_\Lambda
\end{align*}
Moreover, $\pi$ is injective precisely on the subset of non-singular point patterns of $\mathbb{X}_{MS}$.
\end{theo}
\vspace{0.2cm}

By factor map we mean here a continuous, onto and $\mathbb{R}^d$-equivariant map, where on $H\times _{\Sigma} \mathbb{R}^d$ the space $\mathbb{R}^d$ acts through $[w,t]_{\Sigma}.s:= [w,t+s]_{\Sigma}$. The mapping of theorem \ref{theo:parametrization.map} is called the \textit{parametrization map} (\cite{BaaHePl}), and having a window with boundaries of null Haar measure is equivalent to have $\pi $ injective above a full Haar measure subset of $H\times _{\Sigma}\mathbb{R}^d$ (\cite{LeeMo}). As a result we have (see section $2$ for definitions):

\vspace{0.2cm}
\begin{cor}\label{cor:model.set.almost.aut} The dynamical system $(\mathbb{X}_{MS},\mathbb{R}^d)$ associated with a hull of repetitive model multiple sets of $\mathbb{R}^d$ is almost automorphic, with maximal equicontinuous factor given by the compact Abelian group $ H\times _{\Gamma }\mathbb{R}^d$.
\end{cor}

\section{General aspects of dynamical systems}

In this section we provide a short presentation of the notions of equicontinuity, proximality and of Ellis semigroup for general dynamical systems. A complete exposition on this topic can be found in the book \cite{Au}. Along this section we consider a compact dynamical system $(\mathbb{X},T)$, that is, a compact (Hausdorff) space $\mathbb{X}$ with an action of a group $T$ by homeomorphism. We assume moreover that $T$ is Abelian, and for simplicity that $\mathbb{X}$ is metrizable although results also hold by using uniformities. Denote the collection of homeomorphisms issued from the group action by 
\begin{align*}T^*:= \left\lbrace t^* \in Homeo(\mathbb{X}) \, \vert \, t\in T\right\rbrace 
\end{align*}

\vspace{0.1cm}
\begin{de}\label{defequi} A compact dynamical system $(\mathbb{X},T)$ is equicontinuous if the family $T^*$ is equicontinuous, that is,
\begin{align*} \forall \, \varepsilon \; \; \exists \, \delta \; \text{ such that } d(x,x')<\delta \; \Rightarrow \; d(x.t,x'.t)<\varepsilon \; \forall \; t\in T
\end{align*}
\end{de}

Such dynamical systems are minimal whether there are topologically transitive, and in this case are \textit{Kronecker systems}: From the Ascoli theorem the closure $T_\mathbb{X}$ of the group of $T^*$ in $(Homeo(\mathbb{X}), d_\infty )$, $d_\infty$ the distance of uniform convergence on $\mathbb{X}$, is a compact Abelian group acting on $\mathbb{X}$ freely and transitively by homeomorphism. There is an obvious group morphism
\begin{align}\label{general.group.morphism} \begin{psmatrix}[colsep=1.2cm,
rowsep=0cm]
T \; \;  &\; \; T^* \subseteq T_\mathbb{X} 
\psset{arrows=->,linewidth=0.4pt, labelsep=1pt
,nodesep=0pt}
\ncline{1,1}{1,2}
\end{psmatrix}
\end{align}
such that $T$ acts by composition (right or left) on $T_\mathbb{X}$, which we note additively by $\chi \mapsto \chi +t^*$, yielding a compact dynamical system $(T_\mathbb{X}, T)$. Given a transitive point $\mathfrak{o}\in \mathbb{X}$ one has a homeomorphism $T_\mathbb{X}\simeq \mathbb{X}$, $\chi \mapsto \chi (\mathfrak{o})$, which conjugates the $T$ action. Therefore the space $\mathbb{X}$ admits a compact Abelian group structure with $\mathfrak{o}$ as unit and $T$-action \textit{by rotation}, that is, coming from a group morphism of the form (\ref{general.group.morphism}).

\vspace{0.2cm}
To any compact dynamical system is naturally associated an equicontinuous compact dynamical system:
\vspace{0.2cm}
\begin{theo}\label{theoremfactor} Let $(\mathbb{X},T)$ be a compact dynamical system. There exist a unique closed $T$-invariant equivalence relation $\sim _{eq}$ on $\mathbb{X}$, such that the quotient space $\mathbb{X}_{eq} := \mathbb{X} \diagup _{\sim _{eq} }$ with $T$-action is an equicontinuous flow, which is maximal in the sense that any equicontinuous factor of $(\mathbb{X},T)$ factors through $\mathbb{X}_{eq}$.
\end{theo}

\vspace{0.2cm}
This is an existential result, whose proof can be found in \cite{Au}. We generically denote quotient map under the relation $\sim _{eq}$ by
\begin{align*}\begin{psmatrix}[colsep=1.2cm,
rowsep=0cm]
\pi: \mathbb{X} \; \; & \; \; \mathbb{X}_{eq}
\psset{arrows=->,linewidth=0.4pt, labelsep=1pt
,nodesep=0pt}
\ncline{1,1}{1,2}
\end{psmatrix}
\end{align*}
The relation $\sim _{eq}$ is called the \textit{equicontinuous structure relation} and the space $\mathbb{X}_{eq}$ the \textit{maximal equicontinuous factor} of $(\mathbb{X},T)$. If the dynamical system $(\mathbb{X},T)$ is topologically transitive, with some $x _0\in \mathbb{X}$ having dense $T$-orbit, then so is $(\mathbb{X}_{eq},T)$, which is thus a Kronecker system with unit $\mathfrak{o}=\pi (x_0)$. In this case $\mathbb{X}_{eq}$ is a compact Abelian group related to the spectral analysis of $(\mathbb{X},T)$ (see \cite{BaKe}, \cite{BaaLeMo}).

\vspace{0.2cm}
A (algebraic) character $\omega $ on the group $T$ is called an \textit{topological eigenvalue} for the system $(\mathbb{X},T)$ if there exist a non trivial continuous complex-valued function $\chi _\omega $ on $\mathbb{X}$, called an \textit{eigenfunction} associated with $\omega$, satisfying $\chi _\omega (x.t)=\omega (t)\chi _\omega(x)$ for all $x\in \mathbb{X}$ and all $t\in T$. We remark that if $T$ is endowed with a topology so that the action becomes separately continuous then a topological eigenvalue must be a continuous caracter on $T$. The collection of topological eigenvalues forms an Abelian group which we denote $\mathcal{E}(\mathbb{X},T)$. If we let $x_0$ be a transitive point of $(\mathbb{X},T)$ then each topological eigenvalue admits a unique associated eigenfunction $\chi _\omega $ having constant modulus and such that $\chi _\omega (x_0)=1 $ in the unit circle $ \mathbb{T}$.

\vspace{0.2cm}
\begin{prop}\label{prop:eigenvalue}\cite{BaKe} Let $(\mathbb{X},T)$ be a compact topologically transitive dynamical system with $T$ Abelian. Then
\begin{align*} \pi (x)= \pi (x') \qquad \Longleftrightarrow \qquad \chi _\omega (x)=\chi _\omega (x') \; \forall \; \omega \in \mathcal{E}(\mathbb{X},T)
\end{align*}
\end{prop}

\vspace{0.2cm}
What this proposition shows is that the group of topological eigenvalues of $(\mathbb{X},T)$ is naturally isomorphic with the Pontryagin dual of the compact Abelian group $\mathbb{X}_{eq}$,
\begin{align*} \widehat{\mathbb{X}_{eq}} \; \simeq \; \mathcal{E}(\mathbb{X},T)
\end{align*}
where a continuous character $\omega $ on $(\mathbb{X}, T)$ has its eigenfunction $\chi _\omega$ writing as $ \varpi \circ \pi$ with $ \varpi \in \widehat{\mathbb{X}_{eq}}$. In the context where $(\mathbb{X}, T)$ is minimal, one can give a fairly different formulation of the equicontinuous structure relation, that is, whether or not two points $x$ and $x'$ are identified under $\pi$:

\vspace{0.2cm}
\begin{theo}\label{theo:regional.proximality}\cite{Au} Let $(\mathbb{X},T)$ be a minimal compact dynamical system. Then $\pi (x)= \pi(x')$ if and only if $x$ and $x'$ are regionally proximal, that is, for any $\varepsilon >0$ there are $x_\varepsilon $, $x_\varepsilon '$ in $\mathbb{X}$ and $t\in T$ such that
\begin{align*} d(x,x_\varepsilon )<\varepsilon \qquad d(x',x'_\varepsilon )<\varepsilon \qquad d(x_\varepsilon .t,x'_\varepsilon .t)<\varepsilon
\end{align*}
\end{theo}
\vspace{0.2cm}
In this article we are particularly interested in a particular type of dynamical systems, namely the \textit{almost automorphic} systems. In the case where $(\mathbb{X},T)$ is minimal this is equivalent to say that the continuous eigenfunctions associated with $\mathcal{E}(\mathbb{X},T)$ separate \textit{at least} one point of $\mathbb{X}$. 

\vspace{0.2cm}
\begin{de}\label{def:al.aut} \cite{Vee} A compact dynamical system $(\mathbb{X},T)$ is almost automorphic if the factor map $\pi$ admits a one-point fiber.
\end{de}

\vspace{0.2cm}
Even if a dynamical system $(\mathbb{X},T)$ is not equicontinuous, one can still consider the completion $T_\mathbb{X}$ of $T^*$, which is an Abelian group of homeomorphisms on $\mathbb{X}$, but by doing so we loose compacity. Instead, one is turned toward a competion of $T^*$ which is compact, but where the group structure is lost: this is called the \textit{Ellis enveloping semigroup} of $(\mathbb{X},T)$.

\vspace{0.2cm}
\begin{de}\label{def:ellis.compact} Let $(\mathbb{X},T)$ be a compact dynamical system, and consider $T^*$ as a subset of $\mathbb{X}^\mathbb{X}$ the product space with product topology, or equivalently the space of all maps from $\mathbb{X}$ into itself with pointwise convergence topology. The Ellis semigroup $E(\mathbb{X},T)$ of this system is the closure of $ T^*$ in $\mathbb{X}^\mathbb{X}$, endowed with composition of maps.
\end{de}

\vspace{0.2cm}
The Ellis semigroup has a well-defined semigroup structure and is compact (Hausdorff), consisting of transformations on $\mathbb{X}$ obtained as pointwise limits of homeomorphisms in $T^*$. It acts on $\mathbb{X}$, here with action written on the right side, meaning that $x.g$ stands for the image of $x$ under the map $g$, and with this convention it is always a \textit{right-topological semigroup}: if a net $g_\lambda $ converges (pointwise) to $g$ then $h.g_\lambda $ converges (pointwise) to $h.g$. Usually the transformations in $E(\mathbb{X},T)$ are neither continuous, nor invertible on $\mathbb{X}$. Ellis semigroup for dynamical systems is still not completely understood and admits a fairly developed literature (\cite{Au}, \cite{Gl0}, \cite{Gl}, \cite{Gl2}, \cite{GlMe}, \cite{GlMeUs}). However we shall use it only as a tool for proving the results of the present article.

\vspace{0.2cm}
We will also need the notion of Ellis semigroup for locally compact dynamical systems $(\mathbb{X},T)$ which we define below. denote by $\hat{\mathbb{X}}$ the one-point compactification of $\mathbb{X}$ (with neigborhood basis at the point at infinity $\infty $ given by the complementary sets of compacts sets in $\mathbb{X}$), endowed with the extended $T$-action by homeomorphism defined by keeping the point at infinity fixed.

\vspace{0.2cm}
\begin{de}\label{def:ellis.loc} The Ellis semigroup of a locally compact dynamical system $(\mathbb{X},T)$ is defined to be $$E(\mathbb{X},T):= E(\hat{\mathbb{X}},T)\cap \mathcal{F}_{\mathbb{X}}$$
where $\mathcal{F}_{\mathbb{X}}$ is the semigroup of mappings in $\hat{X}^{\hat{\mathbb{\mathbb{X}}}}$ which map $\mathbb{X}$ into itself and keep the point at infinity fixed, with topology induced from $\hat{\mathbb{X}}^{\hat{\mathbb{X}}}$.
\end{de}

\vspace{0.2cm}
We can view $E(\mathbb{X},T)$ as a subsemigroup of $\mathbb{X}^\mathbb{X}$ simply by restricting transformations on $\mathbb{X}$, and the latter gives rise to the same topology on $E(\mathbb{X},T)$. Again it is a right-topological semigroup which contains $T^*$ as a dense subgroup. To show the difference between $E(\mathbb{X},T)$ and $E(\widehat{\mathbb{X}},T)$ we show the following:

\vspace{0.2cm}
\begin{prop}\label{prop: application.infini} Suppose that $(\mathbb{X},T)$ is a non-compact locally compact minimal dynamical system, with $T$ Abelian. Then the transformation $\bowtie _\mathbb{X}$ identically equals to $\infty $ on $\hat{\mathbb{X}}$ lies into $E(\hat{\mathbb{X}},T)$ but not in $E(\mathbb{X},T)$.
\end{prop}
\vspace{0.2cm}
\begin{proof} Let $x$ be some point of $\mathbb{X}$: its $T$-orbit is dense in $\mathbb{X}$, and since this latter is non-compact the point $\infty$ is not isolated, so that the $T$-orbit of $\mathbb{X}$ is in fact dense in $\hat{\mathbb{X}}$. As $\infty $ is invariant under the extended $T$-action on $\hat{\mathbb{X}}$ the only $T$-invariant compact subset of $\hat{\mathbb{X}}$ is the singleton $\left\lbrace \infty \right\rbrace $. From the Zorn lemma, the dynamical system $(E(\hat{\mathbb{X}},T),T)$ with $T$ acting by composition admits some $T$-minimal subsystem $(M,T)$, into which we may choose a transformation $g$. Then for each $x\in \hat{\mathbb{X}}$, the evaluation map $ev_x: E(\hat{\mathbb{X}},T) \longrightarrow \hat{\mathbb{X}}$ is continuous and $T$-equivariant (as $T$ is Abelian, any $t\in T$ comutes with any transformation in $E(\hat{\mathbb{X}},T)$), so maps $M$ onto a minimal subsystem of $\hat{\mathbb{X}}$, which must then be $\left\lbrace \infty \right\rbrace $. It follows that $x.g= \infty $ for each $x\in \hat{\mathbb{X}}$, as desired. It is clear that $\bowtie _\mathbb{X}$ cannot be a transformation in  $E(\mathbb{X},T)$ since it doesn't preserve $\mathbb{X}$ in $\widehat{\mathbb{X}}$, concluding the proof.
\end{proof}

\vspace{0.2cm}
\section{Combinatoric topology on a hull of point patterns}

Given some point pattern $\Lambda $ of $\mathbb{R}^d$, we wish to consider here a topology on its associated hull $\mathbb{X}_{\Lambda }$, which we call \textit{combinatoric topology}, obtained from the metric
\begin{align*} \mathsf{d}(\Lambda ', \Lambda ''):= \inf \left\lbrace \dfrac{1}{1+R} \; \vert \; \Lambda ' \cap B(0,R) = \Lambda ''\cap B(0,R) \right\rbrace 
\end{align*}
This metric is in fact an ultrametric on the hull $\mathbb{X}_{\Lambda }$, setting that two point patterns are close whenever they exactly match on a large domain about the origin in $\mathbb{R}^d$. We refer the space $\mathbf{X}_{\Lambda }$ as the collection $\mathbb{X}_{\Lambda }$ itself, endowed with the above ultrametric, and call it the \textit{combinatoric hull} of $\Lambda $. The action of $\mathbb{R}^d$ by translation site by site on any point patterns yields an action by homeomorphisms on $\mathbf{X}_{\Lambda }$ and thus a dynamical system $(\mathbf{X}_{\Lambda }, \mathbb{R}^d)$. Observe that on the transversal $\Xi $ given in (\ref{canonical.transversal}) the ultrametric $\mathsf{d}$ coincides with the usual metric $d$, and with respect to this new topology $\Xi$ is a \textit{clopen} set, that is, is both closed and open in $\mathbf{X}_{\Lambda }$.

\vspace{0.2cm}
\begin{prop} Let $\Lambda $ be a point pattern of $\mathbb{R}^d$. Then if $\Lambda$ is of finite local complexity then the space $\mathbf{X}_{\Lambda }$ is locally compact. Moreover the dynamical system $(\mathbf{X}_{\Lambda }, \mathbb{R}^d)$ is minimal if and only if $\Lambda$ is repetitive.
\end{prop}

\vspace{0.2cm}
Consider now a point pattern $\Lambda _0$ of finite local complexity with hull $\mathbb{X}_{\Lambda _0}$, and $\mathbb{X}_{eq}$ be the maximal equicontinuous factor of $(\mathbb{X}_{\Lambda _0}, \mathbb{R}^d)$. The combinatoric topology on $\mathbf{X}_{\Lambda _0}$ is naturally finer than the topology of $\mathbb{X}_{\Lambda _0}$, and thus the closed $\mathbb{R}^d$-invariant equivalence relation induced on $\mathbb{X}_{\Lambda _0}$ from the factor map $\pi$ onto $\mathbb{X}_{eq}$ is also closed in $\mathbf{X}_{\Lambda _0}$. One may thus define the space $\mathbf{X}_{eq}$ as the quotient space of $\mathbf{X}_{\Lambda _0}$ under this relation. In other words we have a commuting diagram

\begin{align*} \begin{psmatrix}[colsep=1.5cm,
rowsep=1cm]
\mathbf{X}_{\Lambda _0} \; \;  &  \;  \mathbf{X}_{eq}\\
\mathbb{X}_{\Lambda _0} \; \;  & \;  \mathbb{X}_{eq}&
\psset{arrows=->>,linewidth=0.4pt, labelsep=1pt
,nodesep=0pt}
\ncline{1,1}{1,2}^{\Pi}
\ncline{2,1}{2,2}^{\pi}
\psset{arrows=<->,linewidth=0.4pt, labelsep=0.9pt
,nodesep=0pt}
\ncline {1,1}{2,1}
\ncline {1,2}{2,2}
\end{psmatrix}
\end{align*}
where vertical maps are the identity maps, continuous from top to bottom. The space $\mathbf{X}_{eq}$ admits an induced $\mathbb{R}^d$-action, yielding so a dynamical system $(\mathbf{X}_{eq},\mathbb{R}^d)$. The quotient map, denoted here $ \Pi$, becomes equivariant with respect to the $\mathbb{R}^d$-actions. When $\Lambda _0$ is repetitive the dynamical system is minimal.

\vspace{0.2cm}
Consider the particular element $\mathfrak{o}:= \pi(\Lambda _0)\in \mathbb{X}_{eq}$. The space $\mathbb{X}_{eq}$ is (see section $2$) a compact Abelian group with $\mathfrak{o}$ as unit, and we endow $\mathbf{X}_{eq}$ with this Abelian group structure. We have a (non-continuous) group morphism
\begin{align}\label{group.morphism} \begin{psmatrix}[colsep=1.4cm,
rowsep=0cm]
\mathbb{R}^d \; \; & \; \;  \mathbf{X}_{eq} \\
 \; t \; \; ~ & \; \;  t^*
\psset{arrows=->,linewidth=0.2pt, labelsep=0.8pt
,nodesep=0pt}
\ncline{1,1}{1,2}
 \psset{arrows=|->,linewidth=0.2pt, labelsep=0.7pt
,nodesep=0pt}
\ncline{2,1}{2,2}
\end{psmatrix}
\end{align}
such that the image under the action of a $t\in \mathbb{R}^d$ on an element $\mathfrak{e}\in \mathbf{X}_{eq}$ writes $\mathfrak{e}+ t^*$. The equivariance property of $\Pi$ formulates as $\Pi(\Lambda +t)= \Pi(\Lambda )+t^*$ for any $\Lambda \in \mathbb{X}_{\Lambda _0}$ and any $t\in \mathbb{R}^d$. Then our aim in this section is to show that, under the Meyer and repetitivity assumptions, the topology and the Abelian group structure on $\mathbf{X}_{eq}$ are compatible. Precisely:
\vspace{0.2cm}
\begin{theo}\label{theo:loc.compact} Let $\mathbf{X}_{\Lambda _0}$ be the combinatoric hull of a repetitive Meyer multiple set $\Lambda _0$, and let $\mathbf{X}_{eq}$ be as above. Then $\mathbf{X}_{eq}$ is a locally compact Abelian group.
\end{theo}
\vspace{0.2cm}
The great difficulty in proving this statement is to point out how the Meyer property is related to the quotient topology of $\mathbf{X}_{eq}$. To that end a main tool is proposition \ref{prop:regional.proximality} stated below. First we introduce the regional proximality relation of the system $(\mathbf{X}_{\Lambda _0}, \mathbb{R}^d)$, called after \cite{BaKe} the \textit{strong regional proximality}:
\vspace{0.2cm}
\begin{de}\label{def:strong.regional.proximality} Two multiple sets $\Lambda $ and $\Lambda '$ of a combinatoric hull $\mathbf{X}$ are strongly regionaly proximal if for any given $R>0 $ there exists $\Lambda _R$, $\Lambda '_R$ in $\mathbf{X}$ and $t\in \mathbb{R}^d$ such that one has 
\begin{align*} \Lambda \cap B(0,R) &\equiv \Lambda _R\cap B(0,R)\\
\Lambda ' \cap B(0,R) &\equiv \Lambda '_R\cap B(0,R)\\
\Lambda _R\cap B(t,R) &\equiv \Lambda '_R\cap B(t,R)
\end{align*}
\end{de}
\vspace{0.2cm}
This relation means that, although $\Lambda $ and $\Lambda '$ may not agree anywhere, each always agree with some respective $\Lambda _R$ and $\Lambda '_R$ on a patch of radius $R$ around the origin, and the latters in turns agreeing on some patch of radius $R$. In order to prove theorem \ref{theo:loc.compact} we will make use of the following fact :
\vspace{0.2cm}
\begin{prop}\label{prop:regional.proximality}\cite{BaKe} Let $\Lambda $ and $\Lambda '$ be two multiple sets of the hull of a repetitive Meyer multiple set. Then $\Pi(\Lambda )=\Pi (\Lambda ')$ if and only if $\Lambda$ et $\Lambda '$ are strongly regionaly proximal.
\end{prop}
\vspace{0.2cm}
We add here a simple but useful lemma:

\vspace{0.2cm} 
\begin{lem}\label{lem:differences} Let $\Lambda _0$ be a point pattern of $\mathbb{R}^d$ with hull $\mathbf{X}_{\Lambda _0}$. If two point patterns $\Lambda ,\Lambda '$ in $\mathbf{X}_{\Lambda _0}$ are strongly regionally proximal then $\underline{\Lambda }-\underline{\Lambda }' \subseteq \underline{\Lambda }_0 - \underline{\Lambda } _0 +\underline{\Lambda }_0 - \underline{\Lambda } _0 + \underline{\Lambda }_0 - \underline{\Lambda } _0 $.
\end{lem}

\vspace{0.2cm} 
\begin{proof} Given some point pattern $\Lambda \in \mathbf{X}_{\Lambda _0}$, it is clear since any local configuration of $\Lambda $ appears somewhere in $\Lambda _0$ that $\underline{\Lambda }-\underline{\Lambda } \subseteq \underline{\Lambda }_0-\underline{\Lambda }_0$. Now if we suppose $\Lambda $ and $\Lambda '$ strongly regionally proximal then there exists, for $R_0$ to be some radius of relative density of any multiple set in $\mathbf{X}_{\Lambda _0}$, some $\Lambda _1$ and $\Lambda _2$ together with a vector $t\in \mathbb{R}^d$ such that one can find points $p_1$, $p_2$ et $p_3$ satisfying
\begin{align*} p_1 \in \Lambda \cap B(0,R_0) & \equiv \Lambda _1 \cap B(0,R_0)\\ p_2 \in \Lambda ' \cap B(0,R_0) & \equiv \Lambda _2\cap B(0,R_0)\\ p_3 \in \Lambda _1\cap B(t,R_0) & \equiv \Lambda _2 \cap B(t,R_0)
\end{align*}
It ensure that $\Lambda = \Lambda -p_1 +(p_1-p_3)+(p_3-p_2)+p_2 $ is supported into $ (\underline{\Lambda }-\underline{\Lambda } )+ (\underline{\Lambda }_1-\underline{\Lambda }_1)+ (\underline{\Lambda }_2-\underline{\Lambda }_2) +\underline{\Lambda }'$, yielding the proof.
\end{proof}

\vspace{0.2cm} 
\textbf{Proof of the theorem.} Let us start by supposing that $\Lambda _0$ is a Delone multiple set of finite local complexity in $\mathbb{R}^d$. The space $\mathbf{X}_{\Lambda _0}$ is consequently locally compact and one can form the Ellis semigroup $E(\mathbf{X}_{\Lambda _0}, \mathbb{R}^d)$ (see definition \ref{def:ellis.loc}), which is a right-topological semigroup.

\vspace{0.2cm}
\begin{prop}\label{prop:morphisme.ellis.equi} The mapping given by the composition $\Pi \circ ev_{\Lambda _0}$, $ev_{\Lambda _0}$ denoting the evaluation map at $\Lambda _0$ on the Ellis semigroup $E(\mathbf{X}_{\Lambda _0}, \mathbb{R}^d)$,
\begin{align*}\begin{psmatrix}[colsep=1.5cm,
rowsep=0cm]
\chi : E(\mathbf{X}_{\Lambda _0}, \mathbb{R}^d) \; \; &  \; \; \mathbf{X}_{eq} &
\psset{arrows=->,linewidth=0.2pt, labelsep=1pt
,nodesep=0pt}
\ncline{1,1}{1,2}
\end{psmatrix}  
\end{align*}
is a continuous semigroup morphism.
\end{prop}

\vspace{0.2cm}
\begin{proof} The map $\chi $ is continuous, being the composition of two continuous maps. The image under $\chi$ of the identity map in $E(\mathbf{X}_{\Lambda _0}, \mathbb{R}^d)$ is $\mathfrak{o}$ from the choice of the latter, and for each $t\in \mathbb{R}^d$ and $g\in E(\mathbf{X}_{\Lambda _0}, \mathbb{R}^d)$ we have the equality $\chi (g.t)= \Pi(\Lambda _0.g+t)=  \Pi(\Lambda _0.g)+ t^*= \chi (g)+t^*$. Select then two transformations $g$ and $h$ in $E(\mathbf{X}_{\Lambda _0}, \mathbb{R}^d)$, and consider a net $(t_\lambda )_\lambda \subset \mathbb{R}^d$ converging to $h$ in $E(\mathbf{X}_{\Lambda _0}, \mathbb{R}^d)$ (that is, pointwise on $\mathbf{X}_{\Lambda _0}$). From the right-continuity of the composition law on $E(\mathbf{X}_{\Lambda _0}, \mathbb{R}^d)$ we have on one hand that $\chi (g.h)$ is the limit of $ \chi (g.t_\lambda )= \chi (g)+ t_\lambda^*$ in $\mathbf{X}_{eq}$, and on the other hand that $ \chi (h)= \Pi(\Lambda _0.h)$ is the limit of $\Pi(\Lambda _0+t_\lambda)= t_\lambda^*$ in $\mathbf{X}_{eq}$. But the identity map from $\mathbf{X}_{eq}$ onto $\mathbb{X}_{eq}$ is continuous so these convergences also hold in $\mathbb{X}_{eq}$. Since for this latter the sum is continuous we deduce that $\chi (g)+ t_\lambda^*$ converges to $\chi (g)+ \chi (h)$ in $\mathbb{X}_{eq}$, giving $\chi (g.h)=\chi(g)+\chi (h)$, as desired.
\end{proof}

\vspace{0.2cm}
Consider the one-point compactification $\hat{\mathbf{X}}_{\Lambda _0}$ of  the space$\mathbf{X}_{\Lambda _0}$. It writes as a disjoint union $\mathbf{X}_{\Lambda _0}\cup \left\lbrace \infty \right\rbrace $, with basis of open neighborhoods of $\infty $ given by the complementary sets of the compact subsets of $\mathbf{X}_{\Lambda _0}$. Consider the extended $\mathbb{R}^d$-action by homeomorphisms keeping the point at infinity $\infty $ fixed, and consider the compact right-topological semigroup $E(\hat{\mathbf{X}}_{\Lambda _0}, \mathbb{R}^d)$. This later naturally contains an isomorphic copy of $E(\mathbf{X}_{\Lambda _0}, \mathbb{R}^d)$, as well as the map $\bowtie _{\mathbf{X}_{\Lambda _0}}$ identically equal to $\infty $ on $\hat{\mathbf{X}}_{\Lambda _0}$ by proposition \ref{prop: application.infini}.

\vspace{0.2cm}
\begin{prop}\label{prop:ellis} Suppose that $\Lambda _0$ has the Meyer property. If a transformation $g$ in $E(\hat{\mathbf{X}}_{\Lambda _0}, \mathbb{R}^d)$ gives $\Lambda .g\neq \infty $ for some $\Lambda \in \mathbf{X}_{\Lambda _0}$, then $\Lambda .g\neq \infty $ for any other $\Lambda \in \mathbf{X}_{\Lambda _0}$. In other words we have $$E(\mathbf{X}_{\Lambda _0}, \mathbb{R}^d )= E(\hat{\mathbf{X}}_{\Lambda _0}, \mathbb{R}^d )\backslash \lbrace\bowtie _{\mathbf{X}_{\Lambda _0}}\rbrace $$
\end{prop}

\vspace{0.2cm}
\begin{proof} Consider a transformation $g\in E(\hat{\mathbf{X}}_{\Lambda _0}, \mathbb{R}^d )$ with some $\Lambda _1\in \mathbf{X}_{\Lambda _0}$ with $\Lambda _1.g \neq \infty $, and let $\Lambda _2\in \mathbf{X}_{\Lambda _0}$. We can chose a net $(t_\lambda )_\lambda \subset \mathbb{R}^d $ converging to $g$ in $E(\mathbf{X}_{\Lambda _0}, \mathbb{R}^d )$, which has to satisfy \begin{tabular}{l} $\Lambda _1+t_\lambda \longrightarrow \Lambda _1.g $ \\ $\Lambda _2+t_\lambda \longrightarrow \Lambda _2.g $ \end{tabular} in $\mathbf{X}_{\Lambda _0}$.
Let $r_0$ and $R_0$ be some radii of uniform discretness and relative density of any multiple set of $\mathbf{X}_{\Lambda _0}$. As $\Lambda _1+t_\lambda $ converges to $ \Lambda _1.g $ with respect to the combinatoric metric, there exists an index $\lambda _0$ so that for any $\lambda > \lambda _0$ we may find a common point $$p\in (\Lambda _1+t_\lambda )\cap B(0 , R_0)\equiv (\Lambda _1+t_{\lambda '})\cap B(0, R_0)\equiv (\Lambda _1.g) \cap B(0, R_0)$$
We hence have $t_\lambda - t_{\lambda '}\in \Lambda _1-\Lambda _1$ for any $\lambda , \lambda'>\lambda_0$. Select then a net $v_\lambda $ of points, each taken within the patch $(\Lambda _2+t_\lambda )\cap B(0, R_0)$. Each $v _\lambda$ can be written as $p_\lambda + t_\lambda $ for some point $p_\lambda\in \Lambda _2$. Consider then the collection $\left\lbrace v_\lambda \right\rbrace _{\lambda >\lambda _0}$, lying inside the relatively compact set $B(0,2R_0)$: the difference $v_\lambda -v_{\lambda '}$ is equal to $ (p_\lambda + t_\lambda) -( p_{\lambda '} + t_{\lambda '} )= (p_\lambda - p_{\lambda '})+ (t_\lambda - t_{\lambda '} )$, so falls into $\Lambda _2-\Lambda _2+ \Lambda _1-\Lambda _1\subset \Lambda _0-\Lambda _0+ \Lambda _0-\Lambda _0$, uniformly discrete due to the Meyer property. The collection $\left\lbrace v_\lambda \right\rbrace _{\lambda >\lambda _0}$ is consequently uniformly discrete, and so is finite. Pick up some finite set of representatives $v_{\lambda _1}, ..., v_{\lambda _l}$: the net $( \Lambda _2+t_\lambda )_{\lambda >\lambda _0} $ stay within the compact subset $\bigcup _{i=1}^l \Xi + v_{\lambda _i} $ of $\mathbf{X}_{\Lambda _0}$, where $\Xi $ is the canonical transversal. Thus the limit $\Lambda _2.g$ lies inside $\mathbf{X}_{\Lambda _0}$, and as $\Lambda _2$ has been chosen arbitrary in $\mathbf{X}_{\Lambda _0}$ the statement is proved.
\end{proof}

\vspace{0.2cm}
\begin{cor}\label{cor:big.enough} Suppose that $\Lambda _0$ has the Meyer property. Then the space $E(\mathbf{X}_{\Lambda _0}, \mathbb{R}^d)$ is locally compact and the evaluation map
\begin{align*}\begin{psmatrix}[colsep=1.3cm,
rowsep=0cm]
ev_{\Lambda _0}: E(\mathbf{X}_{\Lambda _0}, \mathbb{R}^d) \; \; &  \; \;\mathbf{X}_{\Lambda _0} &
\psset{arrows=->>,linewidth=0.2pt, labelsep=1pt
,nodesep=0pt}
\ncline{1,1}{1,2}
\end{psmatrix}  
\end{align*}
is onto and proper.
\end{cor}
\vspace{0.2cm}
\begin{proof} The evaluation map $ev_{\Lambda _0}:  E(\hat{\mathbf{X}}_{\Lambda _0}, \mathbb{R}^d) \longrightarrow \hat{\mathbf{X}}_{\Lambda _0} $ from the very definition of the topology on $E(\hat{\mathbf{X}}_{\Lambda _0}, \mathbb{R}^d)$, and is onto since its image is a compact set containing the $\mathbb{R}^d$-orbit of $\Lambda _0$, dense in $\mathbf{X}_{\Lambda _0}$ and consequently dense in $\hat{\mathbf{X}}_{\Lambda _0}$ (the space $\mathbf{X}_{\Lambda _0}$ being never compact, even for $\Lambda _0$ a lattice, the point at infinity $\hat{\mathbf{X}}_{\Lambda _0}$ is always an accumulation point of $\mathbf{X}_{\Lambda _0}$). Any antecedent transformation of the point $\infty$ under the map $ev_{\Lambda _0}$ in $E(\hat{\mathbf{X}}_{\Lambda _0}, \mathbb{R}^d)$ must maps $\Lambda _0$ onto $\infty $, and so has to be $\bowtie _{\mathbf{X}_{\Lambda _0}}$ from the previous proposition. Hence the restriction of $ev_{\Lambda _0}$ outside $\bowtie _{\mathbf{X}_{\Lambda _0}}$ takes its values in $\mathbf{X}_{\Lambda _0}$ and is onto. Properness easily follows.
\end{proof}
\vspace{0.2cm}
\begin{prop}\label{prop:loc.compact} Suppose that $\Lambda _0$ has the Meyer property and is repetitive. Then the space $\mathbf{X}_{eq}$ is locally compact, and $ \Pi$ is a proper map.
\end{prop}
\vspace{0.2cm}
\begin{proof} Let $\sim $ be the closed $\mathbb{R}^d$-invariant equivalence relation induced by $\Pi$ on $\mathbf{X}_{\Lambda _0}$ with graph $\mathcal{G}(\sim )$ in $\mathbf{X}_{\Lambda _0}\times \mathbf{X}_{\Lambda _0}$. This relation coincides with the strong regional proximality by proposition \ref{prop:regional.proximality}.
Define the relation $\hat{\sim }$ on the one-point compactification $\hat{\mathbf{X}}_{\Lambda _0}$ as the relation of graph $\mathcal{G}(\hat{\sim }):= \mathcal{G}(\sim )\cup \left\lbrace (\infty, \infty )\right\rbrace $ in $\hat{\mathbf{X}}_{\Lambda _0}\times \hat{\mathbf{X}}_{\Lambda _0}$. Clearly, $\hat{\sim }$ is an equivalence relation and is moreover $\mathbb{R}^d$-invariant on $\hat{\mathbf{X}}_{\Lambda _0}$. Let us show that it is closed:

\vspace{0.2cm}
Let $\mathcal{G}(\hat{\sim })\supset (\Lambda ^1_{\lambda }, \Lambda ^2_{\lambda })_\lambda \longrightarrow (x,y)$ be some net converging in $\hat{\mathbf{X}}_{\Lambda _0}\times \hat{\mathbf{X}}_{\Lambda _0}$. We may suppose that $(\Lambda ^1_{\lambda }, \Lambda ^2_{\lambda })\in \mathcal{G}(\sim )$, and for a contradiction we may also suppose that the limit pair $(x,y)$ falls outside the graph $\mathcal{G}(\hat{\sim })$. From the very construction of $\hat{\sim }$ we hence have (up to a switch of $x$ and $y$) $x=\Lambda \in \mathbf{X}_{\Lambda _0}$ and $y= \infty $. Let $r_0$ and $R_0$ be some radii of uniform discretness and relative density of any multiple set of $\mathbf{X}_{\Lambda _0}$. From proposition \ref{prop:regional.proximality} any $\Lambda ^1_{\lambda }$ and $ \Lambda ^2_{\lambda }$ are strongly regionaly proximal, and thus by lemma \ref{lem:differences} satisfy
\begin{align}\label{inclusion1} \underline{\Lambda}^1_{\lambda } - \underline{\Lambda}^2_{\lambda } \subseteq 3(\underline{\Lambda }_0 - \underline{\Lambda } _0)
\end{align}
Now as $\Lambda ^1_\lambda$ converges to $x= \Lambda $ in $\mathbf{X}$ there is some index $\lambda _0$ such that for any indices $\lambda , \lambda'>\lambda_0$ there is some common point $p$ within both $ \Lambda ^1_\lambda, \Lambda ^1_{\lambda '}$ and $\Lambda $, and consequently for any $\lambda , \lambda'>\lambda_0$ one obtains 
\begin{align}\label{inclusion2} \underline{\Lambda}^1_{\lambda } - \underline{\Lambda}^1_{\lambda '} =  \underline{\Lambda}^1_{\lambda } - p +p -\underline{\Lambda}^1_{\lambda '} \subseteq (\underline{\Lambda}^1_{\lambda } - \underline{\Lambda}^1_{\lambda }) + (\underline{\Lambda}^1_{\lambda '} - \underline{\Lambda}^1_{\lambda '}) \subseteq 2(\underline{\Lambda }_0 - \underline{\Lambda } _0)
\end{align}
Let then $\left\lbrace v_\lambda \right\rbrace _\lambda$ be a collection of points each contained into $\Lambda ^2_\lambda  \cap B(0, R_0)$. Then by (\ref{inclusion1}) and (\ref{inclusion2})
\begin{align*}v_\lambda-v_{\lambda '} \in \underline{\Lambda}^2_{\lambda } - \underline{\Lambda}^2_{\lambda '} \subseteq \left[ 3(\underline{\Lambda }_0 - \underline{\Lambda } _0) + \underline{\Lambda}^1_{\lambda }\right] - \left[ 3(\underline{\Lambda }_0 - \underline{\Lambda } _0) + \underline{\Lambda}^1_{\lambda '}\right] \subseteq 8(\underline{\Lambda }_0 - \underline{\Lambda } _0)
\end{align*}

the right term being uniformly discrete by proposition \ref{prop:meyer}. The collection $\left\lbrace v_\lambda \right\rbrace _\lambda $ is consequently uniformly discrete and lies inside the precompact set $B(0, R_0)$, so is a finite set. Taking a finite set of representatives $v_{\lambda _1}, ..., v_{\lambda _l}$ we obtain that the net $( \Lambda ^2_\lambda )_\lambda $ remains into the compact subset $\bigcup _{i=1}^l \Xi + v_{\lambda _i} $ of $\mathbf{X}_{\Lambda _0}$. This contradict the fact that $( \Lambda ^2_\lambda )_\lambda $ conververges to $\infty $, giving that $\hat{\sim }$ is a closed relation.\\
Let $Z:= \hat{\mathbf{X}}_{\Lambda _0}/_{\hat{\sim }}$ with $\hat{\pi }_Z: \hat{\mathbf{X}}_{\Lambda _0}\twoheadrightarrow Z$ onto continuous $\mathbb{R}^d$-equivariant map of compact spaces. It is clear that $Z\backslash \left\lbrace \pi _Z(\infty )\right\rbrace $ is locally compact, and the restriction $\pi _Z: \mathbf{X}_{\Lambda _0} \twoheadrightarrow Z\backslash \left\lbrace \pi _Z(\infty )\right\rbrace$ is continuous, onto and proper. But from the construction of $\hat{\sim }$ we have $Z\backslash \left\lbrace \pi _Z(\infty )\right\rbrace$ under bijective correspondance with $\mathbf{X}_{eq}$, so by definition of the quotient topology we get a continuous bijective map $i: \mathbf{X}_{eq}\longrightarrow Z\backslash \left\lbrace \pi _Z(\infty )\right\rbrace $. The map $i$ is proper: for, if $K$ is compact in $\mathbf{X}_{eq}$, its preimage $(\Pi)^{-1}(K)$ writes as $ (\pi _Z)^{-1}(i(K))$, so is compact. It follows that $i$ is an homeomorphism, and thus $\mathbf{X}_{eq}$ is a locally compact space and the map $\Pi$ is proper.
\end{proof}

Now we can establish the proof of theorem \ref{theo:loc.compact}:

\begin{proof} From corrolary \ref{cor:big.enough} together with proposition \ref{prop:loc.compact} we now that the morphism
\begin{align*}\begin{psmatrix}[colsep=1.5cm,
rowsep=0cm]
\chi : E(\mathbf{X}_{\Lambda _0}, \mathbb{R}^d)  & \mathbf{X}_{eq} &
\psset{arrows=->,linewidth=0.4pt, labelsep=1.5pt
,nodesep=0pt}
\ncline{1,1}{1,2}
\end{psmatrix}  
\end{align*}
is onto and proper between locally compact spaces.
Let now $(v_\lambda )_\lambda $ be a net in $\mathbf{X}_{eq}$ converging to some $v$, and consider $w\in \mathbf{X}_{eq}$: we have to show that $v_\lambda +w$ converges to $v+w$ in $\mathbf{X}_{eq}$. This is equivalent to the fact that each subnet $v_{\lambda '}+w$ accumulates at $v+w$. Thus let us consider a subnet $v_{\lambda '}+w$. Then $v_{\lambda '}$ still converges to $v$. If we consider $(g_{\lambda '})_{\lambda '}$ and $h$ to be liftings in $E(\mathbf{X}_{\Lambda _0}, \mathbb{R}^d )$ of the subnet and $w$ respectively, then for any compact neighborhood $V$ of $v$ the net $(g_{\lambda '})_{\lambda '}$ eventually lies into the compact set $\chi ^{-1}(V)$. Consequently we may select an accumulation point $g$ (that is, a transformation lying into $\bigcap _{\lambda _0} \overline{\left\lbrace g_{\lambda '}\right\rbrace _{\lambda '>\lambda _0}}$ contained into $\bigcap _{V\ni v}\chi ^{-1}(V)= \chi ^{-1}(v)$, so that $\chi (g)=v$. Since $(g_{\lambda '})_{\lambda '}$ accumulates at $g$ the net $(h.g_{\lambda '})_{\lambda '}$, as image of the net $(g_{\lambda '})_{\lambda '}$ through the continuous map $s\mapsto h.s$ of $E(\mathbf{X}_{\Lambda _0}, \mathbb{R}^d)$, accumulates at $h.g$. Taking image under the continuous morphism $\chi $ we obtain that $\chi (h.g_{\lambda '})= w+ v_{\lambda '}$ accumulates at $\chi (h.g)=w+v$ in $\mathbf{X}_{eq}$. This shows that addition is right-continuous on $\mathbf{X}_{eq}$, and since it is Abelian, addition is then separately continuous. It implies by \cite{Au}, since the topology on $\mathbf{X}_{eq}$ is locally compact, that addition is jointly continuous and inversion is continuous as well, so that $\mathbf{X}_{eq}$ is a locally compact Abelian group.
\end{proof}

\vspace{0.2cm}
\section{The internal system of a hull of point patterns}

Given a point pattern $\Lambda _0$ of $\mathbb{R}^d$, its generated group is the countable subgroup $\langle \underline{\Lambda }_0\rangle$ of $\mathbb{R}^d$ generated by the support of $\Lambda _0$. It naturally contains the diffrence set $ \underline{\Lambda }_0-  \underline{\Lambda }_0$, and thus the difference set $ \underline{\Lambda }- \underline{\Lambda }$ of any other point pattern $\Lambda \in \mathbb{X}_{\Lambda _0}$. From now on, we shall consider a subgroup 
\begin{align*} \langle \underline{\Lambda }_0\rangle \leqslant \Gamma \leqslant \mathbb{R}^d
\end{align*}
endowed with discrete topology, which we keep fixed. We don't assume any countability assumtion, so in particular it may be the entire Euclidean space.

\vspace{0.2cm}
\begin{de}\label{def:discrete.subsystem} Given a point pattern $\Lambda _0$ with combinatoric hull $\mathbf{X}_{\Lambda _0}$, the subsystem associated with $\Gamma$ is the space with combinatoric metric and restricted $\Gamma$-action
\begin{align*}\Xi ^{\Gamma }:= \left\lbrace \Lambda \in \mathbb{X}_{\Lambda _0} \; \vert \; \underline{\Lambda } \subset \Gamma \right\rbrace 
\end{align*}

\end{de}
\vspace{0.2cm}
It is easy to observe that $\Xi ^\Gamma$ remains stable under the $\Gamma $-action on $\mathbf{X}_{\Lambda _0}$, and is by homeomorphisms, providing so a dynamical system $(\Xi ^\Gamma, \Gamma )$. From the very choice of $\Gamma$ the point pattern $\Lambda _0$ lies in $\Xi ^\Gamma$, and it can be verified that its $\Gamma$-orbit is dense in this latter space. The dynamical system $(\Xi ^\Gamma, \Gamma )$ is minimal if and only if $\Lambda _0$ is repetitive.

\vspace{0.2cm} 
\begin{lem}\label{lem:subsystem} The transversal $\Xi$ is a clopen subset of the space $\Xi ^\Gamma$, which is in turns a clopen subset of the combinatoric hull $\mathbf{X}_{\Lambda _0}$.
\end{lem}

\vspace{0.2cm}

\begin{proof} The fact that $\Xi$ is a clopen subset of the space $\Xi ^\Gamma$ is obvious. Now the space $\Xi ^{\Gamma }$ is closed in $\mathbf{X}_{\Lambda _0}$ since any net of $\mathbf{X}_{\Lambda _0}$ supported on $\Gamma$ and converging for the combinatoric metric admits a limit point pattern also supported on $\Gamma$. On the other hand, for each $\Lambda \in \Xi ^\Gamma$ the set $\Xi + \gamma $, with $\gamma $ some point taken into then support of $\Lambda$, is an open neighborhood of $\Lambda$ in $\mathbf{X}_{\Lambda _0}$ which is included in $\Xi ^\Gamma $: For if $\Lambda '\in \Xi + \gamma $ then $\underline{\Lambda } ' = \underline{\Lambda }'-\gamma + \gamma \subset (\underline{\Lambda }' - \underline{\Lambda }') + \underline{\Lambda } \subset \langle \underline{\Lambda }_0\rangle +\Gamma = \Gamma$. Thus $\Xi ^{\Gamma }$ is also open in $\mathbf{X}_{\Lambda _0}$.
\end{proof}

The following result is more or less folklore: It asserts that under the finite local complexity assumption the hull $\mathbb{X}_{\Lambda _0}$ is the $\mathbb{R}^d$-\textit{suspension} of the subsystem $(\Xi ^{\Gamma }, \Gamma )$:
\vspace{0.2cm}
\begin{prop}\label{prop:suspension} If $\Lambda _0$ is of finite local complexity, there is a topological conjugacy
\begin{align*} \begin{psmatrix}[colsep=1cm,
rowsep=0cm]
 \Xi ^\Gamma  \times _\Gamma \mathbb{R}^d \; \; & \; \; \mathbb{X} \\
\psset{arrows=->>,linewidth=0.2pt, labelsep=1pt
,nodesep=0pt}
\ncline{1,1}{1,2}
 [\Lambda ,t ]_\Gamma \; \; & \; \; \Lambda -t 
 \psset{arrows=|->,linewidth=0.2pt, labelsep=1pt
,nodesep=0pt}
\ncline{2,1}{2,2}
\end{psmatrix}
\end{align*}
where the left hand side is the quotient of $\Xi ^\Gamma  \times  \mathbb{R}^d$ under the diagonal $\Gamma$-action $(\Lambda , t).\gamma := (\Lambda +\gamma , t+\gamma )$, and equiped with an $\mathbb{R}^d$-action via $[\Lambda , t]_\Gamma .t':= [\Lambda , t+t']_\Gamma$.
\end{prop}
\vspace{0.2cm}
\begin{proof} Consider the map $\Xi ^\Gamma  \times\mathbb{R}^d \longrightarrow \mathbb{X}_{\Lambda _0}$ which to any pair $(\Lambda ,t)$ associates $\Lambda -t$. It is clearly continuous, and is $\mathbb{R}^d$-equivariant if on the respective spaces the $\mathbb{R}^d$-actions writes $(\Lambda ,t).t':= (\Lambda , t+t')$ and $\Lambda .t:= \Lambda -t$. Moreover, if for a given $\Lambda \in \mathbb{X}_{\Lambda _0}$ we take a vector $t$ such that $\Lambda +t$ contains $0$ in its support then $\Lambda = (\Lambda +t)-t$ with $\Lambda +t \in \Xi \subset \Xi ^\Gamma$, which shows that this map is onto. It is also invariant under the $\Gamma$- diagonal action, so taking the quotient under this $\Gamma $-action provides the continous, onto and $\mathbb{R}^d$-equivariant map of the statement. This map is 1-to-1: for ,if $\Lambda -t= \Lambda '-t'$ then $t'-t\in \underline{\Lambda }'-\underline{\Lambda } \subset \Gamma - \Gamma = \Gamma $, and thus $[\Lambda ,t ]_\Gamma = [\Lambda +(t'-t),t +(t'-t)]_\Gamma = [\Lambda ',t' ]_\Gamma$.
To show that it is a homeomorphism it suffices to show that the quotient space $\Xi ^\Gamma  \times _\Gamma \mathbb{R}^d$ is compact. Let $\Xi $ be the canonical transversal of the hull $\mathbb{X}_{\Lambda _0}$. It is a compact subset of $\Xi ^\Gamma $, and if $R_0$ stands for some radius of relative density for each multi-point set in $\mathbb{X}_{\Lambda _0}$ then $\Xi ^\Gamma  \times _\Gamma \mathbb{R}^d = \left[ \Xi \times \overline{B}(0,R_0)\right] _\Gamma $: for, any $(\Lambda ,t)\in \Xi ^\Gamma \times \mathbb{R}^d $ has some $\gamma \in \Lambda \cap B(t,R_0)$, yielding $(\Lambda -\gamma , t-\gamma) \in \Xi \times \overline{B}(0,R_0)$ with $[\Lambda -\gamma , t-\gamma]_\Gamma =  [\Lambda ,t ]_\Gamma $. This shows that the space $\Xi ^\Gamma  \times _\Gamma \mathbb{R}^d$ is compact, completing the proof.
\end{proof}

\vspace{0.2cm}

\section{Construction of a cut $\&$ project scheme}

All along this section we assume that $\Lambda _0$ is a repetitive Meyer multiple set of $\mathbb{R}^d$, and we let $\Gamma$ as before. Recall the existence of the group morphism (\ref{group.morphism}) from $\mathbb{R}^d$ into the locally compact Abelian group $\mathbf{X}_{eq}$. We now set the ingredients in order to construct a cut $\&$ project scheme.
 
\vspace{0.2cm}
\begin{prop}\label{prop:internal.group} The set $H:= \Pi (\Xi ^\Gamma)$ is a clopen locally compact subgroup of $\mathbf{X}_{eq}$, such that $H= \overline{\;  \Gamma ^*\; }^{ \mathbf{X}_{eq}}$.
\end{prop}

\vspace{0.2cm}
\begin{proof} Let us show first that $\Xi ^{\Gamma } $ is saturated over the set $H$ with respect to the mapping $\Pi$: For if $\Lambda $ lies in $\Xi ^\Gamma$ then any $\Lambda '$ with $\Pi (\Lambda ) = \Pi (\Lambda ')$ must be strongly regionally proximal to $\Lambda $ from proposition \ref{prop:regional.proximality}, and thus by lemma \ref{lem:differences} one gets that $\underline{\Lambda }' \subset 3(\underline{\Lambda }_0 - \underline{\Lambda }_0) + \underline{\Lambda } \subset \langle \underline{\Lambda }_0\rangle +\Gamma = \Gamma$. This shows that $\Lambda ' \in \Xi ^\Gamma$, as wished. Now as the space $\Xi ^\Gamma$ is clopen in $\mathbf{X}_{\Lambda _0}$ we obtains, by definition of the quotient topology, that $H$ is also clopen in $\mathbf{X}_{eq}$. Since $\Xi ^\Gamma$ is $\Gamma$-invariant and $\Pi$ is $\mathbb{R}^d$-equivariant we thus obtain that $H$ is $\Gamma$-invariant in $\mathbf{X}_{\Lambda _0}$ as well. As $\Lambda _0$ is supposed repetitive the $\Gamma$-action is minimal on $\Xi ^\Gamma$, and thus the translation action of $\Gamma$ on $H$ is also minimal. Since $\Lambda _0 \in \Xi ^\Gamma$ from the very choice of $\Gamma$, the neutral element $\mathfrak{o} = \Pi (\Lambda _0)$ of the group $\mathbf{X}_{\Lambda _0}$ lies in $H$. It comes by minimality that $H= \overline{\;  \mathfrak{o}.\Gamma \; }^{ \mathbf{X}_{eq}}= \overline{\;  \Gamma ^*\; }^{ \mathbf{X}_{eq}}$. Therefore $H$ is the closure of a subgroup of $\mathbf{X}_{eq}$, and thus is itself a locally compact Abelian group.
\end{proof}

\vspace{0.2cm}
From this process we have extracted a locally compact Abelian group which will be the internal space for our construction of cut $\&$ project scheme. Note that if we consider two different groups with $\Gamma \leqslant \Gamma '$ then we get two different internal spaces, say $H^\Gamma$ and $H^{\Gamma '}$, in $\mathbf{X}_{eq}$. It is clear from the previous proposition that $H^\Gamma$ appears as a clopen subgroup of $H^{\Gamma '}$. The two extreme cases about this construction are when $\Gamma = \langle \underline{\Lambda }_0\rangle$ with associated locally compact Abelian group $H^\mathfrak{o}$, and when $\Gamma = \mathbb{R}^d$ which in this case gives rise to the full locally compact Abelian group $\mathbf{X}_{eq}$. Each resulting $H^\Gamma$ is then interpolated by these two groups.

\vspace{0.2cm}



\begin{de}\label{def*-map} The restriction $\Gamma \longrightarrow H$ of the morphism (\ref{group.morphism}) on $\Gamma $ is the *-map associated to $\Gamma$.
\end{de}

\vspace{0.2cm}
From proposition \ref{prop:internal.group} the *-map as defined here takes its values in $H$ and has dense range in this latter. We will generically write $\gamma$ for an element of $\Gamma$, as well as $\gamma ^*$ for its image under the *-map. Having constructed an internal space and a $*$-map as well, we now need to find out a window in this internal space. This may be done by considering for each index $i\in I$ the associated transversal $\Xi _i:= \left\lbrace \Lambda \in \mathbb{X}_{\Lambda _0} \, \vert \, 0\in \Lambda _i \right\rbrace $ of the hull. Each is compact and open in $\mathbf{X}_{\Lambda _0}$, and is contained in the canonical transversal $\Xi$ and thus in the subsystem $\Xi ^\Gamma$. We thus consider the compact subsets of $H$
\begin{align*} W_i:= -\Pi (\Xi _i) \qquad \forall \; i\in I
\end{align*}


\begin{theo}\label{theo:construction.CPS} The data $(H, \Sigma , \mathbb{R}^d)$ with associated diagram
\begin{align*}\begin{psmatrix}[colsep=1.5cm,
rowsep=0cm]
H \; \; & \; \; H \times \mathbb{R}^d  \; \; & \; \; \mathbb{R}^d \\
\cup \; \; & \; \;  \cup \; \; & \; \; \cup \\
\Gamma ^* \; \; & \; \; \Sigma  \; \; & \; \; \Gamma 
\psset{arrows=->>,linewidth=0.2pt, labelsep=1.3pt
,nodesep=0pt}
\ncline{1,2}{1,3}
\psset{arrows=<->,linewidth=0.2pt, labelsep=1.3pt
,nodesep=0pt}
\ncline{3,2}{3,3}
\psset{arrows=<<-,linewidth=0.2pt, labelsep=1.3pt
,nodesep=0pt}
\ncline{1,1}{1,2}
\ncline{3,1}{3,2}
\end{psmatrix}\end{align*}
where $\Sigma := \left\lbrace (\gamma ^*, \gamma ) \in H \times \mathbb{R}^d \; \vert \, \gamma \in \Gamma \right\rbrace $, is a cut $\&$ project scheme, with quotient $H \times _{\Sigma }\mathbb{R}^d$ isomorphic, as compact group, with $\mathbb{X}_{eq}$. Moreover, the family $\left\lbrace  W_i \right\rbrace _{i\in I}$ is a window in $H$.
\end{theo}

\vspace{0.2cm}
\textbf{Proof of the theorem.} The map $\Pi$ restricts in a mapping
\begin{align}\label{restriction} \begin{psmatrix}[colsep=1.3cm,
rowsep=0cm]
\Pi ^\Gamma : \Xi ^\Gamma \; \; & \; \; H 
\psset{arrows=->,linewidth=0.2pt, labelsep=1.2pt
,nodesep=0pt}
\ncline{1,1}{1,2}
\end{psmatrix}
\end{align} 
which is continuous and proper as $\Pi$ was, $\Gamma$-equivariant, and onto since $\Pi$ was onto and $\Xi ^\Gamma$ is saturated over $H$. Since it is the restriction of $\Pi$, one has $\Pi ^\Gamma(\Lambda ) = \Pi ^\Gamma(\Lambda ')$ if and only if $\Lambda $ and $\Lambda '$ are strongly regionally proximal.

\vspace{0.2cm}
\begin{lem}\label{lem:parametrization.cps} The group $H\times _{\Sigma } \mathbb{R}^d$ obtained as the quotient of $H \times \mathbb{R}^d$ by the subgroup $\Sigma $ is a compact Abelian group conjugated with $\mathbb{X}_{eq}$ through the isomorphism
\begin{align*} \begin{psmatrix}[colsep=1cm,
rowsep=0cm]
H \times _{\Sigma} \mathbb{R}^d \; \; & \; \; \mathbb{X}_{eq} \\
\left[ w,t\right] _{\Sigma} \; \; & \; \; w-t^*
\psset{arrows=<->,linewidth=0.2pt, labelsep=0.8pt
,nodesep=0pt}
\ncline{1,1}{1,2}
\ncline{2,1}{2,2}
\end{psmatrix}
\end{align*}
\end{lem}

\begin{proof} The mapping $\Pi ^\Gamma$ naturally provides a factor map from $\Xi ^\Gamma \times _{\Gamma} \mathbb{R}^d $ onto $H \times _{\Sigma} \mathbb{R}^d$ by letting $[\Lambda , t]_\Gamma$ mapped onto $[\Pi ^\Gamma (\Lambda ),t]_\Sigma$. In addition the map from $H \times \mathbb{R}^d$ to $\mathbb{X}_{eq}$ mapping $(w,t)$ onto $w-t^*$ is well defined, and an onto continuous group morphism which is obviously $\Sigma$-invariant, and thus defines a continuous group morphism from $H \times _{\Sigma} \mathbb{R}^d$ onto $\mathbb{X}_{eq}$. This shows that $H \times _{\Sigma} \mathbb{R}^d$ is a compact (Hausdorff) Abelian group. Now observe that from the conjugacy of $\mathbb{X}_{\Lambda _0}$ with $\Xi ^\Gamma \times _{\Gamma} \mathbb{R}^d $ of proposition \ref{prop:suspension} one has $H \times _{\Sigma} \mathbb{R}^d$ as an equicontinuous factor of $\mathbb{X}_{\Lambda _0}$ an thus must be also a factor of $\mathbb{X}_{eq}$. With this argument one can show that the morphism from $H \times _{\Sigma} \mathbb{R}^d$ onto $\mathbb{X}_{eq}$ is $1$-to-$1$, setting the proof.
\end{proof}

\vspace{0.2cm}
Our wish now is to show that the compact subsets $W_i$ have non-empty interior in $H$, and that they are in fact the closure of their interior. Our strategy is to show that there are point patterns $\Lambda \in \Xi ^\Gamma$ where the mapping $\Pi ^\Gamma$ is \textit{open} (recall that a map $\pi : X \longrightarrow Y$ is open at $x\in X$ if any neighborhood of $x\in X$ is mapped onto a neighborhood of $\pi (x)\in Y$). If we find such point pattern $\Lambda$ within for instance $\Xi _i$, which is an open subset of $\Xi ^\Gamma$, then $\Pi (\Lambda )$ will automatically remains into the interior of $\Pi (\Xi _i)=-W_i$, showing the non-emptiness of this latter interior and thus the non-emptiness of the interior of $W_i$. The next proposition asserts that such point patterns does exists:

\vspace{0.2cm}
\begin{prop}\label{prop:pi.open}\cite{Vee} Let  $(\mathbb{X},\mathbb{R}^d)$ a minimal dynamical system over $\mathbb{X}$ a compact metric space, and $\pi : \mathbb{X} \twoheadrightarrow \mathbb{X}_{eq}$ the factor map onto its maximal equicontinuous factor. Then there exists a dense residual $\mathbb{R}^d$-invariant subset $\mathbb{X}^0_{eq} \subseteq \mathbb{X}_{eq}$ such that $\pi $ is open at any point of the saturated $\mathbb{R}^d$-invariant subset $\mathbb{X}^0:= \pi ^{-1}(\mathbb{X}^0_{eq})\subset \mathbb{X}$.
\end{prop}

\vspace{0.2cm}
\begin{prop}\label{prop:Pi.open} For $\Lambda _0$ a repetitive Meyer multiple set of $\mathbb{R}^d$, the factor map $\pi : \mathbb{X}_{\Lambda _0}\twoheadrightarrow \mathbb{X}_{eq}$ is open exactly where $\Pi : \mathbf{X}_{\Lambda _0}\twoheadrightarrow \mathbf{X}_{eq}$ is.
\end{prop}

\vspace{0.2cm}
\begin{proof} Due to the conjugacy maps of proposition \ref{prop:suspension} and of lemma \ref{lem:parametrization.cps}, one has, when applied to the case $\Gamma = \mathbb{R}^d$ so that we have $\Xi ^\Gamma = \mathbf{X}_{\Lambda _0}$ and $H = \mathbf{X}_{eq}$ in this case, a comutative diagram
\begin{align*}\begin{psmatrix}[colsep=1.5cm,
rowsep=1cm]
\mathbf{X}_{\Lambda _0} \; \; & \; \; \mathbf{X}_{\Lambda _0} \times \mathbb{R}^d  \; \; & \; \; \mathbf{X}_{\Lambda _0} \times _\Gamma \mathbb{R}^d \; \; & \; \; \mathbb{X}_{\Lambda _0} \\
\mathbf{X}_{eq} \; \; & \; \; \mathbf{X}_{eq} \times \mathbb{R}^d  \; \; & \; \; \mathbf{X}_{eq} \times _\Gamma \mathbb{R}^d \; \; & \; \; \mathbb{X}_{eq}
\psset{arrows=<->,linewidth=0.2pt, labelsep=1.3pt
,nodesep=0pt}
\ncline{1,3}{1,4}^{~\; \; \; \; \simeq }
\ncline{2,3}{2,4}^{~ \; \; \;\;  \simeq }
\psset{arrows=->>,linewidth=0.2pt, labelsep=1.3pt
,nodesep=0pt}
\ncline{1,2}{1,3}
\ncline{2,2}{2,3}
\ncline{1,1}{2,1}<{\Pi ~}
\ncline{1,2}{2,2}<{\Pi \times id ~}
\ncline{1,3}{2,3}<{\Pi \times _{\Gamma } id ~}
\ncline{1,4}{2,4}<{\pi ~}
\psset{arrows=<<-,linewidth=0.2pt, labelsep=1.3pt
,nodesep=0pt}
\ncline{1,1}{1,2}
\ncline{2,1}{2,2}
\end{psmatrix}\end{align*}
Now the right side horizontal maps are conjugacies, and the center horizontal maps are open maps, as quotient maps through group actions. Obviously the left side horizontal maps are open maps, and it easily follows that the right side vertical map $\pi$ is open exactly where the left side map $\Pi$ is.
\end{proof}

\vspace{0.2cm}
As a consequence of the two previous propositions the factor map $\Pi ^\Gamma : \Xi ^\Gamma \twoheadrightarrow H$ admits a dense residual $\Gamma$-invariant subset $\Xi ^{\Gamma , 0} \subseteq \Xi ^\Gamma$, which is saturated and with dense image $H^0 \subseteq H$, of elements where $\Pi ^\Gamma$ is open. This is because $\Xi ^\Gamma $ is saturated over $ H$ and that both of these are open in their respective space $\mathbf{X}_{\Lambda _0}$ and $\mathbf{X}_{eq}$.

\vspace{0.2cm}

\begin{prop}\label{prop:image.reguliere} Each $W_i$ is the closure of its interior in $H$. Moreover, the non-empty subset $H^0\subseteq H$ is included into $NS^H:= H\backslash \left[ \Gamma ^* -  \bigcup _{i\in I}\partial W_i \right]$.
\end{prop}

\vspace{0.2cm}
\begin{proof} Given an open subset $U$ of $\Xi ^\Gamma $, let us show the equalities
\begin{align}\label{égalités}\Pi ^\Gamma(U\cap \Xi ^{\Gamma , 0})= int(\Pi  ^\Gamma(U))\cap H^0= \Pi ^\Gamma(U)\cap H^0
\end{align}
For, each $\Lambda \in U\cap \Xi ^{\Gamma , 0}$ admits $U$ as neighborhood, so have an image $\Pi (\Lambda) \in H^0$ admiting $\Pi (U)$ as neighborhood in $H$, giving $\Pi(U\cap \Xi ^{\Gamma , 0})\subset int(\Pi  ^\Gamma(U))\cap H^0$. The inclusion of this latter into $\Pi ^\Gamma(U)\cap H^0$ is obvious, and in turns, any element $w$ of $\Pi ^\Gamma(U)\cap H^0$ admits an antecedent element $\Lambda$ in $U$, necessarily within $\Xi ^{\Gamma , 0}$ as this latter is saturated over $H^0$, which gives $\Pi ^\Gamma(U)\cap H^0 \subset \Pi ^\Gamma(U\cap \Xi ^{\Gamma , 0})$.

\vspace{0.2cm}
Let then $K$ be a compact and topologically regular subset of $\Xi ^\Gamma$: $int(K)\cap \Xi ^{\Gamma , 0}$ is dense in $K$, so the image $\Pi ^\Gamma (int(K))\cap H^0\subset int(\Pi(K))\cap H^0\subset int(\Pi(K))$ is dense in $\Pi (K)$. In particular, $int(\Pi(K))$ is dense in $\Pi (K)$ so this latter isthe closure of its interior. This is in particular true for $K= \Xi _i$ and $\Xi$, as they are both open and closed in $\Xi ^\Gamma$, and it follows that $W_i$ and $W$ are compact subsets which are the closure of their interior in $H$.

\vspace{0.2cm}
Now let us show that $H^0$ lies into $NS^\Gamma $: if $w$ lies into the complementary set of $NS^\Gamma$ in $H$, then it exactly means that $w-\gamma ^*$ lies within $ -\partial W_i$ for some $\gamma \in \Gamma$ and some $i\in I$. There exists then $\Lambda $, with $\Pi (\Lambda )= w$, such that $\Lambda -\gamma \in \Xi _i$. But $\Lambda $ cannot be in $\Xi ^{\Gamma , 0}$ because otherwise $\Lambda -\gamma$ would also be into $\Xi ^{\Gamma , 0}$, and because $\Lambda -\gamma \in \Xi _i$ we would have $w-\gamma ^*\in int(-W_i)=-int(W_i) $, which contradict the choice of $w-\gamma ^*\in -\partial W_i$. Therefore the complementary set of $NS^\Gamma$ lies into the complementary set of $H^0$ in $H^\Gamma$, giving the desired inclusion.
\end{proof}

\vspace{0.2cm}
Now, the theorem \ref{theo:construction.CPS} will be proven once we show the last following lemma. The key argument here, added to the fact that $W$ is topologicaly regular, is the fact that $\Pi $ is a proper map (proposition \ref{prop:loc.compact}).

\vspace{0.2cm}
\begin{lem}\label{lem:discret.dense} $\Sigma := \left\lbrace (\gamma ^*, \gamma ) \in H \times \mathbb{R}^d \; \vert \, \gamma \in \Gamma \right\rbrace $ is a discrete subgroup of $H \times \mathbb{R}^d$.
\end{lem}

\vspace{0.2cm}
\begin{proof}First observe that the $\Gamma$-translates of the canonical transversal $\Xi$ forms an open cover of the subsystem $\Xi ^\Gamma$. Hence the preimage $(\Pi ^\Gamma )^{-1}(-W)$, with $-W$ a compact subset of $H$ and $\Pi ^\Gamma$ being proper, is covered by a finite number of translates of $\Xi$, say $\Xi + F$ where $F\subset \Gamma$ is finite. Let now $\Lambda \in \Xi ^{\Gamma , 0}$ be chosen, which can be selected in the open subset $ \Xi \subset \Xi ^\Gamma$ and for which we denote $w:=\Pi ^\Gamma (\Lambda )\in H^0\subseteq H$ for simplicity. Since $\Lambda $ is a Delone set anf $F$ is finite the difference subset $\underline{\Lambda } -F$ of $\mathbb{R}^d$ admits a radius of uniform discreteness $\delta >0$, and on the other hand $w\in - \mathring{W}$, so that we can form the open neighborhood of $(0,0)$ in $H\times \mathbb{R}^d$ given by $(\mathring{W}+w)\times B(0,\delta )$.

\vspace{0.2cm}
Now suppose that there is another $(\gamma ^*, \gamma)\in \Sigma$ in $(\mathring{W}+w)\times B(0,\delta )$: Then $\gamma ^* \in \mathring{W}+w$ and thus $w-\gamma ^*$ lies into $-W$, yielding that $\Lambda -\gamma $ lies into $ (\Pi ^\Gamma )^{-1}(-W)$. Thus there is some $f\in F$ such that $\Lambda -\gamma \in \Xi +f$, so that $\Lambda -(\gamma +f) \in \Xi$, which means that $\gamma +f$ lies into the support of $\Lambda$. This in turns means that $\gamma$ remains into $\underline{\Lambda } -F$, and from the choice of $\delta$ we must get $\gamma =0$, so that $(\gamma ^*, \gamma)=(0,0)$. This shows that $(0,0)$ is isolated in $\Sigma$, which is consequently discrete in $H \times \mathbb{R}^d$.
\end{proof}

\vspace{0.1cm}
\section{Embedding of repetitive Meyer multiple sets into model multiple sets}

We shall now consider the family of repetitive model multiple sets arising from the cut $\&$ project scheme $(H , \Sigma,\mathbb{R}^d)$ and window $\left\lbrace W_i \right\rbrace _{i\in I}$ given in theorem \ref{theo:construction.CPS}, as well as the connection with the point pattern $\Lambda _0$ we started with. Let us denote $\mathbb{X}_{MS}$ for the hull of repetitive model multiple sets coming from the data provided by theorem \ref{theo:construction.CPS}, which we will refer $\mathbb{X}_{MS}$ as \textit{the hull of repetitive model multiple sets associated with} $\Lambda _0$. The window $\left\lbrace W_i \right\rbrace _{i\in I}$ in $H$ which is provided in theorem \ref{theo:construction.CPS} is not necessarily irredundent, that is, the compact subgroup 
\begin{align*} \mathcal{R}:= \lbrace w\in H \, \vert \, W_i+w=W_i \, \forall \, i\in I \rbrace
\end{align*}
of $H$ may not be trivial. This unable us to parameterize the hull $\mathbb{X}_{MS}$ by the compact Abelian group $H \times _{\Sigma }\mathbb{R}^d$, but as discussed in section $1$ one equally obtains the same hull $\mathbb{X}_{MS}$ of point patterns by considering the cut $\&$ project scheme where the internal space arise as the quotient group $H_\mathcal{R}$ of $H$ by $\mathcal{R}$. In this way the window become irredundant, and by an application of theorem \ref{theo:parametrization.map} one has a parameterization map
\begin{align*}
\begin{psmatrix}[colsep=1.5cm,
rowsep=0cm]
\pi _{MS}: \mathbb{X}_{MS}\; \; & \; \;  H_\mathcal{R}\times _{\Sigma } \mathbb{R}^d 
\psset{arrows=->>,linewidth=0.2pt, labelsep=1.5pt
,nodesep=0pt}
\ncline{1,1}{1,2}
\end{psmatrix}
\end{align*}
which is injective on the subset of non-singular point patterns of $\mathbb{X}_{MS}$.

\vspace{0.2cm}
\begin{theo}\label{theo:inclusion.in.model.set} Let $\Lambda _0$ be a repetitive Meyer multiple sets of $\mathbb{R}^d$ with hull $\mathbb{X}_{\Lambda _0}$, and let $\mathbb{X}_{MS}$ be its associated hull of repetitive model multiple sets. Then

\vspace{0.2cm}
\begin{flushleft}
$(i)$ Any point pattern $\Lambda$ of $ \mathbb{X}_{\Lambda _0}$ admits a $\Delta$ in $\mathbb{X}_{MS}$ such that $\Lambda _i\subseteq \Delta _i $ for each $ i\in I$.

\vspace{0.2cm}
$(ii)$ For residually many model multiple set $\Delta $ of $\mathbb{X}_{MS}$ there exist a strong regional proximality class $ C_\Delta $ included in the saturated subset $\mathbb{X}_{\Lambda _0}^0 \subseteq \mathbb{X}_{\Lambda _0}$ such that
\begin{align*} \Delta _i= \bigcup_{\Lambda \in C_\Delta } \Lambda _i
\end{align*}
$(iii)$ The group of topological eigenvalues $\mathcal{E}(\mathbb{X}_{MS},\mathbb{R}^d)$ is a subgroup of $\mathcal{E}(\mathbb{X}_{\Lambda _0},\mathbb{R}^d)$ and the quotient of the latter by the former is given by
\begin{align*} \widehat{\mathcal{R}}= \mathcal{E}(\mathbb{X}_{\Lambda _0},\mathbb{R}^d)\diagup \mathcal{E}(\mathbb{X}_{MS},\mathbb{R}^d)
\end{align*}
\end{flushleft}
\end{theo}

\vspace{0.1cm}
\begin{proof} Consider first a point pattern $\Lambda \in  \Xi ^{\Gamma ,0}$. For such a $\Lambda$, the element $w_\Lambda := \Pi ^\Gamma (\Lambda )$ lies in $H^0$, in turns included into $NS^H:= H\backslash \left[ \Gamma ^*- \bigcup _{i\in I}\partial W_i \right]$ by proposition \ref{prop:image.reguliere}. Thus $w_\Lambda$ is a non-singular position in $H$ and gives rise to a unique model multiple set in $\mathbb{X}_{MS}$, given for each $i\in I$ by
\begin{align*}
\mathfrak{P}(\mathring{W}_i+w_\Lambda)= \mathfrak{P}(W_i+w_\Lambda)= \left\lbrace \gamma \in \Gamma \, \vert \, \gamma ^*\in W_i+w_\Lambda \right\rbrace 
\end{align*}
Now any $\gamma \in \Lambda _i$ yields $\Lambda -\gamma \in \Xi _i$, so that $w_\Lambda -\gamma ^*\in -W_i$ and consequently $\gamma ^* \in W_i+ w_\Lambda$. This means that $\Lambda _i\subseteq \mathfrak{P}(W_i+w_\Lambda)$ for each $i\in I$, that is, for such $\Lambda $ the collection $\Psi (\Lambda )$ as defined in the statement contains $(\mathfrak{P}(W_i+w_\Lambda))_{i\in I}$ and is thus a non-empty family.

\vspace{0.2cm}
Let now $\Lambda $ be any point pattern in $\mathbb{X}_{\Lambda _0}$. By repetitivity of $\Lambda _0$ and thus by minimality of $(\mathbf{X}_{\Lambda _0}, \mathbb{R}^d)$ we can select some $\Lambda '\in \Xi ^{\Gamma ,0}$ as well as a sequence of vectors $t_n\in \mathbb{R}^d$ such that $\Lambda ' -t_n$ converges to $\Lambda $ with respect to the combinatoric topology. We showed that there exists some model multiple set $\Delta \in \mathbb{X}_{MS}$ containing $\Lambda '$: so if $\gamma $ is taken into $\Lambda _i$ then it lies within $\Lambda '_i -t_n$, and thus within $\Delta _i-t_n$ upon some rank, which shows that any accumulation point of $\Delta -t_n$ in the compact space $\mathbb{X}_{MS}$, which does exists, must contains $\Lambda $ index by index of $I$.

\vspace{0.2cm}
Let us rapidly prove $(iii)$: The parametrization map of the hull $\mathbb{X}_{MS}$ is over $H _{\mathcal{R}} \times _{\Sigma }\mathbb{R}^d$, for which there is the short exact sequence of compact Abelian groups
\begin{align*}\begin{psmatrix}[colsep=1.1cm,
rowsep=0cm]
0 \; \; & \; \; \mathcal{R} \; \;  & \; \; H \times _{\Sigma }\mathbb{R}^d \; \; & \; \; H _{\mathcal{R}} \times _{\Sigma }\mathbb{R}^d \; \; & \; \;  0
\psset{arrows=->,linewidth=0.4pt, labelsep=1pt
,nodesep=0pt}
\ncline{1,1}{1,2}
\ncline{1,2}{1,3}
\ncline{1,3}{1,4}
\ncline{1,4}{1,5}
\end{psmatrix}  
\end{align*}  
which dualizes in a short exact sequence of discrete groups
\begin{align*}\begin{psmatrix}[colsep=1.1cm,
rowsep=0cm]
0 \; \;  & \; \; \mathcal{E}(\mathbb{X}_{MS},\mathbb{R}^d) \; \;  & \; \;\mathcal{E}(\mathbb{X}_{\Lambda _0},\mathbb{R}^d) \; \; & \; \; \widehat{\mathcal{R}} \; \;  & \; \;  0
\psset{arrows=->,linewidth=0.4pt, labelsep=1pt
,nodesep=0pt}
\ncline{1,1}{1,2}
\ncline{1,2}{1,3}
\ncline{1,3}{1,4}
\ncline{1,4}{1,5}
\end{psmatrix}  
\end{align*} 

Now to prove $(ii)$ recall the existence of a dense residual subset $\mathbb{X}_{eq}^0$ such that $\pi$ is open at any point pattern over this subset. One has $\mathbb{X}_{eq} = H \times _{\Sigma }\mathbb{R}^d$ by theorem \ref{theo:construction.CPS} and the image of $\mathbb{X}_{eq}^0$ in $H _{\mathcal{R}} \times _{\Sigma }\mathbb{R}^d$ under the quotient map by $\mathcal{R}$, which is an open map, is then a dense residual subset. On the other hand the subset of non-singular positions in $H _{\mathcal{R}} \times _{\Sigma }\mathbb{R}^d$ is also dense residual, and thus intersect the former subset in a dense residual subset of $H _{\mathcal{R}} \times _{\Sigma }\mathbb{R}^d$. This corresponds to a dense residual sub-collection of non-singular model multiple sets of $\mathbb{X}_{MS}$. Pick up $\Delta $ one one these: its parameter $\pi _{MS} (\Delta ) = [[w]_\mathcal{R},t]_{\Sigma _\mathcal{R}}$ in $H_\mathcal{R} \times _{\Sigma _\mathcal{R}}\mathbb{R}^d$ is with $\mathcal{R}$-coset $[w]_\mathcal{R}$ included in $H^0$. Let $\Delta ':= \Delta +t$. Then this latter is given symbol-wise by
\begin{align*} \Delta '_i = \mathfrak{P}([W_i]_\mathcal{R}+[w]_\mathcal{R}) = \mathfrak{P}(W_i+w+\mathcal{R})= \mathfrak{P}(W_i+w)
\end{align*}
where $w$ is any representative of the coset $[w]_\mathcal{R}$ in $H^0\subseteq H$. Now the argument at the beginning of this proof shows that any $\Lambda \in \mathbb{X}_{\Lambda _0}$ with $\Pi ^\Gamma (\Lambda )= w $ is symbol-wise contained in this model multiple set. On the other hand, the condition $\gamma ^*\in W_i+w$ is equivalent to $w -\gamma ^* \in -W_i$, and since $-W_i = \Pi ^\Gamma (\Xi _i)$ there must exists a $\Lambda $ with $\Pi ^\Gamma (\Lambda )= w$ such that $\Lambda -\gamma \in \Xi _i$, that is, with $\gamma \in \Lambda _i$. From this we deduce that 
\begin{align*}\Delta '_i = \bigcup_{\Pi ^\Gamma(\Lambda )=[w, 0]_\Sigma} \Lambda _i \quad \text{and thus} \quad \Delta _i = \bigcup_{\pi(\Lambda )=[w, t]_\Sigma} \Lambda _i
\end{align*}
Now the set $\pi (\Lambda )=[w, t]_\Sigma$ is a strongly proximal class in $\mathbb{X}_{\Lambda _0}^0$, which finish the proof.


\end{proof}
\vspace{0.2cm}
The point $(ii)$ shows that point patterns in $\mathbb{X}_{MS}$ are independent of the choice of the group $\Gamma$ used to construct the cut $\&$ project scheme $(H,\Sigma ,\mathbb{R}^d)$ and window $\left\lbrace W_i \right\rbrace _{i\in I}$. Moreover, if one starts with a hull of repetitive model multiple sets then the resulting hull $\mathbb{X}_{MS}$ is exactly the original one.
\vspace{0.3cm}
\section{Characterization of model multiple sets by almost automorphic dynamical systems}

We apply here the result of theorem \ref{theo:inclusion.in.model.set} in the particular case of a hull with almost automorphic $\mathbb{R}^d$-action. We show in this context that the hull $\mathbb{X}_{\Lambda _0}$ precisely arise as the hull of repetitive model multiple sets formed from the cut $\&$ project scheme and window given in theorem \ref{theo:construction.CPS}. It leads to the following dynamical characterization of repetitive model multiple sets:
\vspace{0.2cm}
\begin{theo}\label{theo:principal} Let $\Lambda _0$ be a repetitive Meyer multiple set of $\mathbb{R}^d$. Then $\Lambda$ is a model multiple set if and only if the dynamical system $(\mathbb{X}_{\Lambda _0}, \mathbb{R}^d)$ is almost automorphic.
\end{theo}
\vspace{0.2cm}
\begin{proof} It is proven in \cite{Vee} for an almost automorphic dynamical system the subset of points where the factor map $\pi$ onto its maximal equicontinuous factor is $1$-to-$1$ corresponds to the saturated subset $\mathbb{X}_{\Lambda _0}^0$ where $\pi$ is open. Let us pick a $\Delta \in \mathbb{X}_{MS}$ such that point $(ii)$ of theorem \ref{theo:inclusion.in.model.set} holds, that is, for which there is a strong proximality class $C_\Delta$ in $\mathbb{X}_{\Lambda _0}^0$ such that
\begin{align*} \Delta _i= \bigcup_{\Lambda \in C_\Delta } \Lambda _i
\end{align*}
As $\pi$ is injective on $\mathbb{X}_{\Lambda _0}^0$ the class $C_\Delta $ must consists of a single point pattern $\Lambda \in \mathbb{X}_{\Lambda _0}$, yielding $\Lambda = \Delta$. Minimality of $(\mathbb{X}_{\Lambda _0}, \mathbb{R}^d)$ and $(\mathbb{X}_{MS}, \mathbb{R}^d)$ ensure that $\mathbb{X}_{\Lambda _0}= \mathbb{X}_{MS}$, that is, $\Lambda _0$ is a repetitive model multiple set of $\mathbb{R}^d$.
\end{proof}

\vspace{0.2cm}

We wish to end this article with a comment on the Meyer property in the case of uncolared point patterns. In \cite{KeSa} Kellendonk and Sadun investigated whether if a point pattern with hull conjugated with the hull of a Meyer set is itself a Meyer set. It turned out to be false, but still they showed the following:

\vspace{0.2cm}
\begin{theo}\label{theo:deformations} \cite{KeSa} Let $\Lambda _0$ be a repetitive Delone set of finite local complexity of $\mathbb{R}^d$ with hull $\mathbb{X}_{\Lambda _0}$. Then the dynamical system $( \mathbb{X}_{\Lambda _0}, \mathbb{R}^d)$ admits $d$ independent topological eigenvalues in $\mathbb{R}^d$ if and only if for each $\varepsilon >0$ there exists a repetitive Meyer set $\Lambda _\varepsilon $ of $\mathbb{R}^d$ such that:

\vspace{0.2cm}
$(i)$ The dynamical systems $( \mathbb{X}_{\Lambda _0}, \mathbb{R}^d)$ and $( \mathbb{X}_{\Lambda _\varepsilon}, \mathbb{R}^d)$ are conjugated.

$(ii)$ The Hausdorff distance between $\Lambda _0$ and $\Lambda _\varepsilon$ is less than $\varepsilon$.
\end{theo}

\vspace{0.2cm}
If a repetitive Delone sets with finite local complexity admits an almost automorphic system $( \mathbb{X}_{\Lambda _0}, \mathbb{R}^d)$ then it admits $d$ independent topological eigenvalues in $\mathbb{R}^d$: For, the $\mathbb{R}^d$-action on $\mathbb{X}_{\Lambda _0}$ is locally free by uniform discreteness of any point pattern of this hull, and thus must be locally free on the maximal equicontinuous factor in case of almost automorphy. From \cite{BaKe} (lemma $2.22$ there) one has by dualisation an injective morphism
\begin{align*}\begin{psmatrix}[colsep=1.1cm,
rowsep=0cm]
 \mathcal{E}(\mathbb{X}_{\Lambda _0},\mathbb{R}^d)= \widehat{\mathbb{X}_{eq}} \; \;  & \; \;\widehat{\mathbb{R}^d} = \mathbb{R}^d 
\psset{arrows=->,linewidth=0.4pt, labelsep=1pt
,nodesep=0pt}
\ncline{1,1}{1,2}
\end{psmatrix}  
\end{align*} 
with relatively dense range in $\mathbb{R}^d$. Thus $\mathcal{E}(\mathbb{X}_{\Lambda _0},\mathbb{R}^d)$ as a relatively dense subgroup of $\mathbb{R}^d$ contains $d$ linearly independent eigenfunctions. Consequently we can provide a statement for repetitive Delone sets with finite local complexity:

\vspace{0.2cm}
\begin{theo}\label{theo:deformations2} Let $\Lambda _0$ be a repetitive Delone set of finite local complexity of $\mathbb{R}^d$ with hull $\mathbb{X}_{\Lambda _0}$. Then the dynamical system $( \mathbb{X}_{\Lambda _0}, \mathbb{R}^d)$ is almost automorphic if and only if for each $\varepsilon >0$ there exists a repetitive model set $\Lambda _\varepsilon $ of $\mathbb{R}^d$ such that:

\vspace{0.2cm}
$(i)$ the dynamical systems $( \mathbb{X}_{\Lambda _0}, \mathbb{R}^d)$ and $( \mathbb{X}_{\Lambda _\varepsilon}, \mathbb{R}^d)$ are conjugated.

$(ii)$ The Hausdorff distance between $\Lambda _0$ and $\Lambda _\varepsilon$ is less than $\varepsilon$.
\end{theo}

\section*{Acknowledgments}
I wish to express my gratitude to my advisor Johannes Kellendonk for its important comments concerning this article. I also wish to thank Daniel Lenz for the valuable conversations we had on this topic.

\vspace{1cm}
\begin{minipage}{0.5\linewidth}
\begin{flushleft}
\textit{Jean-baptiste Aujogue\\
~ \\
Université de Lyon\\
CNRS UMR 5208\\
Université Lyon 1\\
Institut Camille Jordan\\
43 blvd. du 11 novembre 1918\\
F-69622 Villeurbanne cedex\\
France}

\end{flushleft}
\end{minipage}
\hfill
\begin{minipage}{0.5\linewidth}
\begin{flushleft}
\textit{Current Adress:\\
 ~ \\
Universidad de Santiago de Chile\\
Dep. de Matem\'aticas, Fac. de Ciencia\\
Aladema 3363, Estaci\'on Central\\
Santiago\\
Chile}

\end{flushleft}
\end{minipage}

\end{document}